\documentclass[reqno,12pt]{amsart}
\usepackage{bbm}
\usepackage[all,pdf]{xy}
\usepackage{epsfig}
\usepackage{amsmath}
\usepackage{amssymb}
\usepackage{amscd}
\usepackage{graphicx}
\usepackage{color}

\usepackage{mathptmx}

\allowdisplaybreaks[4]

\makeatletter
\@namedef{subjclassname@2020}{%
	\textup{2020} Mathematics Subject Classification}
\makeatother
\usepackage[colorlinks=true]{hyperref}
\usepackage{amsfonts}

\usepackage[textwidth=17cm, textheight=24.5cm, centering]{geometry}
\usepackage[nameinlink,noabbrev,nosort]{cleveref}

\newtheorem{thm}{Theorem}[section]
\newtheorem{lemma}[thm]{Lemma}

\newtheorem{prop}[thm]{Proposition}

\newtheorem{defn}[thm]{Definition}

\Crefname{thm}{Theorem}{Theorems}
\Crefname{lemma}{Lemma}{Lemmas}
\Crefname{prop}{Proposition}{Propositions}
\Crefname{subsection}{Subsection}{Subsections}

\newcommand{\norm}[1]{\left\Vert #1\right\Vert}

\def \N {\mathbb N}
\def \C {\mathbb C}
\def \Z {\mathbb Z}
\def \R {\mathbb R}
\def \Q {\mathbb Q}

\def \X {\mathcal{X}}
\def \Y {\mathcal{Y}}

\def \t {\mathcal{T}}

\def \d {\delta}

\numberwithin{equation}{section}

\begin{document}
\hypersetup{linkcolor=red}
	
\title[]{Pointwise convergence of double ergodic averages along certain non-polynomial sequences}

	\author[]{Rongzhong Xiao}		
	\address[Rongzhong Xiao]{School of Mathematical Sciences, 
		University of Science and Technology of China, 
		Hefei, Anhui, 230026, PR China}
	\email{xiaorz@mail.ustc.edu.cn}
	

\begin{abstract}
	Fix $c\in (1,2)$. Let $\alpha$ and $\beta$ be two non-zero real numbers. It is shown that for any measure preserving system $(X,\X,\mu,T)$ and any $f,g\in L^{\infty}(\mu)$, the limit
	\begin{equation*}
	\lim_{N\to\infty}\frac{1}{N}\sum_{n=1}^{N}f(T^{\lfloor \alpha n^c \rfloor}x)g(T^{\lfloor \beta n^c \rfloor}x)
	\end{equation*}
	exists for $\mu$-a.e. $x\in X$.
\end{abstract}

\maketitle 
\section{Introduction}
By {\bf measure preserving system}, we mean a tuple $(X,\X,\mu,T)$, where $(X,\X,\mu)$ is a Lebesgue probability space and $T:X\to X$ is an invertible measure preserving transformation. 

Fix a measure preserving system $(X,\X,\mu,T)$. In the study of ergodic theory, a fundamental problem is to understand the pointwise limit behavior of the following double ergodic averages:
\begin{equation}\label{eq1-0}
	\frac{1}{N}\sum_{n=1}^{N}T^{\lfloor P(n)\rfloor}f\cdot T^{\lfloor Q(n)\rfloor}g,
\end{equation}
where $P,Q\in \R[n]$ with $\deg P=\deg Q$ and $f,g\in L^{\infty}(\mu)$.

In 1990, Bourgain established the following pointwise ergodic theorem.
\begin{thm}\label{thm1-1}
	$($\cite[Main Theorem]{Bourgain90}$)$ Let $(X,\X,\mu,T)$ be a measure preserving system. Fix two distinct non-zero integers $a$ and $b$. Then for any $f,g\in L^{\infty}(\mu)$, we have that 
	\begin{equation}\label{eq1-1}
		\lim_{N\to\infty}\frac{1}{N}\sum_{n=1}^{N}f(T^{an}x)g(T^{bn}x)
	\end{equation}
	exists for $\mu$-a.e. $x\in X$. 
\end{thm}
In 2025, Krause extended Bourgain's theorem to linear polynomials with one real coefficient and one rational coefficient. 
\begin{thm}\label{thm1-2}
	$($\cite[Theorem 1.1]{Krause2025}$)$ Let $(X,\X,\mu,T)$ be a measure preserving system. Let $\alpha\in \R\backslash \Q$ and $\beta\in \Q$. Then for any $f,g\in L^{\infty}(\mu)$, we have that 
	$$\lim_{N\to\infty}\frac{1}{N}\sum_{n=1}^{N}f(T^{\lfloor \alpha n\rfloor}x)g(T^{\lfloor \beta n\rfloor}x)$$ exists for $\mu$-a.e. $x\in X$.
\end{thm}
Following \Cref{thm1-1,thm1-2}, the simplest case of \eqref{eq1-0} may be the following double ergodic averages:
\begin{equation}\label{eq1-7}
		\frac{1}{N}\sum_{n=1}^{N}T^{n^2}f\cdot T^{-n^2}g,
\end{equation} 
where $f,g\in L^{\infty}(\mu)$.

Unfortunately, \eqref{eq1-7} seems to be out of reach at this stage due to the absence of distinct degrees. However, in this direction, Daskalakis established a pointwise ergodic theorem for the $\lfloor n^c\rfloor$-analogue of \eqref{eq1-7}.
\begin{thm}\label{thm2-1}
	$(${\em a special case of} \cite[Theorem 1.7]{Daskalakis2025}$)$ Let $c\in (1,23/22)$. Let $(X,\X,\mu,T)$ be a measure preserving system. Then for any $f,g\in L^{\infty}(\mu)$, we have that 
	$$\lim_{N\to\infty}\frac{1}{N}\sum_{n=1}^{N}f(T^{\lfloor n^c\rfloor}x)g(T^{-\lfloor   n^c\rfloor}x)=\lim_{N\to\infty}\frac{1}{N}\sum_{n=1}^{N}f(T^{n}x)g(T^{-n}x)$$ for $\mu$-a.e. $x\in X$.
\end{thm}
In fact, the function $t^c$ can be replaced by a general $c$-regularly varying function (see \cite[Theorem 1.7]{Daskalakis2025}) in Daskalakis's result. Furthermore, Daskalakis's result shows that the double ergodic averages along $(\lfloor n^c\rfloor,-\lfloor n^c\rfloor)$ share similarities with those in \eqref{eq1-1}. Very recently, for the almost everywhere convergence part of \Cref{thm2-1}, Fornal and Krause \cite[a special case of Proposition 1.5]{FornalKrause2026} extended the range of $c$ to $(1,7/6)$. 

Here, we establish a result that combines aspects of \Cref{thm1-2,thm2-1}.
\begin{thm}\label{TB}
Fix $c\in (1,2)$. Let $\alpha$ and $\beta$ be two non-zero real numbers. Let $(X,\X,\mu,T)$ be a measure preserving system. Then for any $f,g\in L^{\infty}(\mu)$, the limit
\begin{equation}\label{eq1-3}
	\lim_{N\to\infty}\frac{1}{N}\sum_{n=1}^{N}f(T^{\lfloor \alpha n^c \rfloor}x)g(T^{\lfloor \beta n^c \rfloor}x)
\end{equation}
exists for $\mu$-a.e. $x\in X$. 
\end{thm}
When $c\in (1,23/22)$ and $|\alpha|=|\beta|$, \Cref{TB} follows from \cite[Theorem 1.7]{Daskalakis2025}; if $c\in (1,7/6)$ and $\alpha,\beta \in \{1,-1\}$, it follows from \cite[Proposition 1.5]{FornalKrause2026}.

\medskip
 \noindent\textbf{Structure of the paper.}
Section~2 reduces \Cref{TB} to \Cref{thm3-1}. Section~3 proves
\Cref{thm3-1}.
The appendices record some estimates used in the proof of \Cref{thm3-1}.

\medskip
\noindent\textbf{Notation.}\label{sec1:notation}
Write $\N=\{1,2,\ldots\}$ and
$\Z_{\ge0}=\{0,1,\ldots\}$. For $x\in\R$, $\lfloor x\rfloor$
is the greatest integer not exceeding $x$,
$\lceil x\rceil=-\lfloor-x\rfloor$, and
$\{x\}=x-\lfloor x\rfloor$. We use
\[
 \begin{gathered}
 e(t)=\exp(2\pi it),\qquad
 \|t\|_{\R/\Z}=\min_{m\in\Z}|t-m|\quad(t\in\R),\\
 \widehat a(\xi)=\int_\R a(u)e(-\xi u)\,du\quad(a\in L^1(\R),\ \xi\in\R).
 \end{gathered}
\]
The notation $A\lesssim_{d_1,\ldots,d_k}B$ means
$A\le C_{d_1,\ldots,d_k}B$ for a positive constant depending only on
the indicated data; $A\asymp_{d_1,\ldots,d_k}B$ means both inequalities.
The notations $O_{d_1,\ldots,d_k}(B)$ and $o_{N\to\infty}(1)$ have their
usual meanings.

For an open set $\mathcal O\subset\R$ and
$\mathbb K\in\{\R,\C\}$, $C^r(\mathcal O;\mathbb K)$ denotes
functions with continuous derivatives through order $r$ on
$\mathcal O$. If $I=[a,b]\subset\mathcal O$ is a closed bounded
interval, we put
$\|f\|_{C^r(I)}=\max_{0\le j\le r}\sup_I|f^{(j)}|$;
the derivatives are restrictions from $\mathcal O$.
A function in $C_c^\infty(\mathcal O;\mathbb K)$ has support compactly
contained in $\mathcal O$ and is extended by zero when used on $\R$.
For a pointwise-defined function $a:I\to\C$, its total variation and
variation norm are
\[
 \operatorname{Var}_I(a)=
 \sup_{\substack{s\in\N,\ t_0<\cdots<t_s\\t_0,\ldots,t_s\in I}}
       \sum_{j=1}^s|a(t_j)-a(t_{j-1})|,
 \qquad
 \mathcal V_I(a)=\sup_I|a|+\operatorname{Var}_I(a).
\]
The same definition is used on $\R$; variation on a singleton is zero.

\medskip
\noindent\textbf{Acknowledgements.}
The author is supported by the National Natural Science Foundation of China (No. 12426201), and by the Postdoctoral Fellowship Program of CPSF under Grant Number GZB20260746.

\medskip
\noindent\textbf{Declaration on the use of AI.}
The mathematical results, arguments, and proofs in this paper 
were developed by the author, who takes full responsibility for the final manuscript.
AI tools were used only for language polishing and to assist with literature searches . 
\section{Reduction to \texorpdfstring{\Cref{thm3-1}}{Theorem~\getrefnumber{thm3-1}}}
\begin{defn}
	Let $d\in \N$. A tuple $(X,\X,\mu, (T^{{\bf t}})_{{\bf t}\in \R^{d}})$ is called a {\bf measurable flow} if $(X,\X,\mu )$ is a Lebesgue probability space, $(T^{\bf t})_{{\bf t}\in \R^{d}}$ is a $d$-parameter group of invertible measure preserving transformations acting on $X$, and the mapping $$\R^d\times X\rightarrow X, \ ({\bf t}, x)\mapsto T^{{\bf t}}x$$ is measurable.
\end{defn}
Fix a measure preserving system $(X,\X,\mu,T)$, $1$-bounded $f,g\in L^{\infty}(\mu)$, $c\in (1,2)$, and two non-zero real numbers $\alpha,\beta$. Set
\[
 X_0=\bigcap_{j\in\Z}
 \{x\in X:|f(T^jx)|\le1\text{ and }|g(T^jx)|\le1\}.
\]
The set $X_0$ is $T$-invariant and co-null. Redefining $f$ and $g$ to be
zero outside $X_0$, we may assume that $|f(x)|,|g(x)|\le1$ for every $x\in X$.

First, let us construct a measurable flow based on $(X,\X,\mu,T)$.  Let $Y=X\times [0,1)$, $\nu=\mu\times \text{Leb}_{[0,1)}$, and $\Y=\X\otimes \mathcal{B}([0,1))$. For any $t\in \R$, we define $$S^{t}:Y\to Y,(x,s)\mapsto (T^{\lfloor t+s\rfloor}x,t+s\mod{1}).$$ This defines a measurable flow ${\bf Y}:=(Y,\Y,\nu,(S^t)_{t\in \R})$. For any $(x,s)\in Y$, let $\tilde{f}(x,s)=f(x)$ and $\tilde{g}(x,s)=g(x)$. 

Next, we introduce the following averages:
$$D_{N}(\tilde{f},\tilde{g})(y)=\frac{1}{N}\sum_{n=1}^{N}\tilde{f}(S^{\alpha n^c}y)\tilde{g}(S^{\beta n^c}y);$$
$$D_{N}(f,g)(x)=\frac{1}{N}\sum_{n=1}^{N}f(T^{\lfloor \alpha n^c \rfloor}x)g(T^{\lfloor \beta n^c \rfloor}x);$$
$$C_N(\tilde{f},\tilde{g})(y)=\frac{1}{N}\int_{0}^{N}\tilde{f}(S^{\alpha t^c}y)\tilde{g}(S^{\beta t^c}y)dt.$$

\subsection{Lift the convergence to flows}
\begin{lemma}\label[lemma]{lem2-1}
	If for $\nu$-a.e. $y\in Y$, the limit $$\lim_{N\to\infty}D_{N}(\tilde{f},\tilde{g})(y)$$ exists, then for $\mu$-a.e. $x\in X$, the limit 
	$$\lim_{N\to\infty}D_{N}(f,g)(x)$$
	exists.
\end{lemma}
The proof of \Cref{lem2-1} uses \Cref{lem2-2}.
\begin{lemma}\label[lemma]{lem2-2}
	$($\em{Cf.} \cite[Theorem 1.3]{Boshernitzan1994}$)$ Let $c\in (1,2)$. Then for any non-zero real number $\alpha$, the sequence $\alpha n^c$ is equidistributed modulo one. 
\end{lemma}
\begin{proof}[Proof of \Cref{lem2-1}]
For each $s\in[0,1)$ and $z\in\R$, the integers $\lfloor z+s\rfloor$
and $\lfloor z\rfloor$ agree unless $\{z\}\in[1-s,1)$. Therefore, for
any $(x,s)\in X_0\times[0,1)$, we have that
\begin{align}
 |D_N(\tilde f,\tilde g)(x,s)-D_N(f,g)(x)|
\lesssim \frac1N\sum_{n=1}^N
 \left(1_{[1-s,1)}(\{\alpha n^c\})
       +1_{[1-s,1)}(\{\beta n^c\})\right).
 \label{eq2-add-height-error}
\end{align}
\Cref{lem2-2} implies
\begin{equation}
 \limsup_{N\to\infty}
 |D_N(\tilde f,\tilde g)(x,s)-D_N(f,g)(x)|\lesssim s.
 \label{eq2-add-height-limsup}
\end{equation}
For every $j\ge2$, we can choose $s_j\in(0,1/j)$ and a co-null
measurable set $X_j\subset X_0$ such that $D_N(\tilde f,\tilde g)(x,s_j)$
converges for every $x\in X_j$. Fix $\displaystyle x\in\bigcap_{j\ge2}X_j$.
For a fixed $j>1$,
\begin{align*}
 |D_N(f,g)(x)-D_M(f,g)(x)|
 &\le |D_N(f,g)(x)-D_N(\tilde f,\tilde g)(x,s_j)|\\
 &\hspace{0.5cm}+|D_N(\tilde f,\tilde g)(x,s_j)
             -D_M(\tilde f,\tilde g)(x,s_j)|\\
 &\hspace{1cm}+|D_M(\tilde f,\tilde g)(x,s_j)-D_M(f,g)(x)|.
\end{align*}
The middle term tends to zero as $N,M\to\infty$. Taking the joint
limit superior in the other terms and applying
\eqref{eq2-add-height-limsup}, we obtain
\[
 \limsup_{N,M\to\infty}|D_N(f,g)(x)-D_M(f,g)(x)|\lesssim s_j.
\]
Letting $j\to\infty$ proves \Cref{lem2-1}.
\end{proof}
\subsection{Convergence of the continuous averages}
\begin{lemma}\label[lemma]{lem2-3}
	For $\nu$-a.e. $y\in Y$, the limit $$\lim_{N\to\infty}C_{N}(\tilde{f},\tilde{g})(y)$$ exists.
\end{lemma}
The proof of \Cref{lem2-3} uses \Cref{thm2-2}
and \Cref{lem2-4}.
\begin{thm}\label{thm2-2}
	$($\cite[Theorem 2.9]{HSX2023}$)$ For $\nu$-a.e. $y\in Y$, the limit $$\lim_{N\to\infty}\frac{1}{N}\int_{0}^{N}\tilde{f}(S^{\alpha t}y)\tilde{g}(S^{\beta t}y)dt$$ exists.
\end{thm}
Here, \cite[Theorem 2.9]{HSX2023} is applied to the flow
$(S^{\alpha t})_{t\in\R}$ with the ratio $\beta/\alpha$; see also
\cite[Theorem 8.30]{BLM2012} for the pair like $t,2t$.
\begin{lemma}\label[lemma]{lem2-4}
$($\emph{a special case of }\cite[Lemma 7.2]{FN2023}$)$ 	Let $\delta>0$. Let $f\in L^{\infty}(m_{\R})$, where $m_{\R}$ is the Lebesgue measure on $\R$. If the limit $\displaystyle \lim\limits_{M\to\infty}\frac{1}{M}\int_{0}^{M}f(t)dt$ exists,
then the limit
$$\lim\limits_{M\to\infty}\frac{1}{M}\int_{0}^{M}f(t^\delta)dt$$ exists and the two limits are equal.
\end{lemma}

\begin{proof}[Proof of \Cref{lem2-3}]
\Cref{thm2-2} and \Cref{lem2-4} imply
\Cref{lem2-3}.
\end{proof}
\subsection{Reduction to orbitally smooth functions}
First, we introduce a special function family.
\begin{defn}
	Given $F\in L^{\infty}(\nu)$, it is called \textbf{orbitally smooth} if it has a bounded representative such that
	$t\mapsto F(S^t y)$ is smooth and each derivative is bounded uniformly in $t,y$. We denote all orbitally smooth $1$-bounded measurable functions with respect to $(Y,\Y,\nu,(S^t)_{t\in \R})$ by $C^{\infty}_{1}({\bf Y})$. 
\end{defn}
\begin{lemma}\label[lemma]{lem2-6}
	If for all $F,G\in C^{\infty}_{1}({\bf Y})$, $$\lim_{N\to\infty}\left|D_{N}(F,G)(y)-C_{N}(F,G)(y)\right|=0$$ for $\nu$-a.e. $y\in Y$, then for $\nu$-a.e. $y\in Y$,
	$$\lim_{N\to\infty}\left|D_{N}(\tilde{f},\tilde{g})(y)-C_{N}(\tilde{f},\tilde{g})(y)\right|=0.$$
\end{lemma}
\begin{proof}
Choose a nonnegative function $\kappa\in C_c^\infty((-1,1))$ with
$\displaystyle\int_\R\kappa(t)\,dt=1$. For $0<\delta<1/4$, set
\[
 \kappa_\delta(t)=\delta^{-1}\kappa(t/\delta),\qquad
 F_\delta(y)=\int_\R\kappa_\delta(t)\tilde f(S^t y)\,dt,
 \quad
 G_\delta(y)=\int_\R\kappa_\delta(t)\tilde g(S^t y)\,dt.
\]
Positivity and normalization of the kernel imply that both functions
are $1$-bounded. To verify orbital smoothness, write
\[
 F_\delta(S^v y)
 =\int_\R\kappa_\delta(t-v)\tilde f(S^t y)\,dt.
\]
For each integer $r\ge0$,
\[
 \sup_{v,y}\left|\frac{d^r}{dv^r}F_\delta(S^v y)\right|
 \le\|\kappa_\delta^{(r)}\|_{L^1(\R)}
 =\delta^{-r}\|\kappa^{(r)}\|_{L^1(\R)}.
\]
The same argument applies to $G_\delta$. Thus
$F_\delta,G_\delta\in C_1^\infty({\bf Y})$.

Let $E_\delta=[0,\delta]\cup[1-\delta,1)$. If
$\delta\le s\le1-\delta$ and $|t|<\delta$, then
$\lfloor s+t\rfloor=0$. Consequently,
$F_\delta(x,s)=\tilde f(x,s)$ and
$G_\delta(x,s)=\tilde g(x,s)$ at those points. Then
\begin{equation}
 |F_\delta-\tilde f|(x,s)\lesssim1_{E_\delta}(s),\qquad
 |G_\delta-\tilde g|(x,s)\lesssim1_{E_\delta}(s).
 \label{eq2-add-smoothing-boundary}
\end{equation}
For each fixed $s$ and each nonzero real $\gamma$, \Cref{lem2-2} gives
\begin{equation}
 \lim_{N\to\infty}\frac1N\sum_{n=1}^N
 1_{E_\delta}(\{s+\gamma n^c\})=2\delta.
 \label{eq2-add-boundary-discrete-count}
\end{equation}
Observe that the periodic function $u\mapsto1_{E_\delta}(\{s+\gamma u\})$
has continuous Ces\`aro limit $2\delta$. Applying \Cref{lem2-4}, with the power equal to $c$, gives
\begin{equation}
 \lim_{N\to\infty}\frac1N\int_0^N
 1_{E_\delta}(\{s+\gamma t^c\})\,dt=2\delta.
 \label{eq2-add-boundary-continuous-count}
\end{equation}
By \eqref{eq2-add-smoothing-boundary} and
\eqref{eq2-add-boundary-discrete-count},
\eqref{eq2-add-boundary-continuous-count} with $\gamma=\alpha,\beta$,
\begin{align}
 \limsup_{N\to\infty}
 |D_N(F_\delta,G_\delta)-D_N(\tilde f,\tilde g)|
 &\lesssim\delta,\label{eq2-add-smoothing-discrete}\\
 \limsup_{N\to\infty}
 |C_N(F_\delta,G_\delta)-C_N(\tilde f,\tilde g)|
 &\lesssim\delta.\label{eq2-add-smoothing-continuous}
\end{align}
Choose a countable sequence $\delta_j\downarrow0$. By the hypothesis
of \Cref{lem2-6}, there is a common co-null set on which
\[
 D_N(F_{\delta_j},G_{\delta_j})
 -C_N(F_{\delta_j},G_{\delta_j})\longrightarrow0
 \quad\text{for every }j.
\]
At a point of this set,
\eqref{eq2-add-smoothing-discrete} and \eqref{eq2-add-smoothing-continuous}
give
\[
 \limsup_{N\to\infty}
 |D_N(\tilde f,\tilde g)-C_N(\tilde f,\tilde g)|\lesssim\delta_j
 \quad\text{for every }j.
\]
Letting $j\to\infty$ proves \Cref{lem2-6}.
\end{proof}
\subsection{Reduction to \texorpdfstring{\Cref{thm3-1}}{Theorem~\getrefnumber{thm3-1}}}
To state \Cref{lem2-5}, for a fixed $\phi \in C^{\infty}_{c}((0,\infty))$ and $F,G\in C^{\infty}_{1}({\bf Y})$, define
$$D_{N}^{\phi}(F,G)(y)=\frac{1}{N}\sum_{n\ge 1}\phi(n/N)F(S^{\alpha n^c}y)G(S^{\beta n^c}y);$$
$$C_N^{\phi}(F,G)(y)=\frac{1}{N}\int_{0}^{\infty}\phi(t/N)F(S^{\alpha t^c}y)G(S^{\beta t^c}y)dt.$$
\begin{lemma}\label[lemma]{lem2-5}
	Suppose that for all $F,G\in C^{\infty}_{1}({\bf Y})$ and all $\phi \in C^{\infty}_{c}((0,\infty))$, there is some $\delta_c>0$ such that for each $N>1$, 
	$$\norm{D_{N}^{\phi}(F,G)-C_N^{\phi}(F,G)}_{L^{2}(\nu)}\lesssim_{F,G,\phi,c,\alpha,\beta} N^{-\delta_c}.$$
	Then for all $F,G\in C^{\infty}_{1}({\bf Y})$, $$\lim_{N\to\infty}\left|D_{N}(F,G)(y)-C_{N}(F,G)(y)\right|=0$$ for $\nu$-a.e. $y\in Y$.
\end{lemma}
\begin{proof}[Proof of \Cref{lem2-5}]
Fix $F,G\in C_1^\infty({\bf Y})$. First fix a real number $q>1$ and put
$N_j=\lfloor q^j\rfloor$, omitting the finitely many initial repetitions.
For each fixed $\phi$, the assumed power saving implies
\[
 \sum_j \norm{D_{N_j}^\phi(F,G)-C_{N_j}^\phi(F,G)}_{L^2(\nu)}^2<\infty.
\]
Therefore,
\begin{equation}
 D_{N_j}^\phi(F,G)(y)-C_{N_j}^\phi(F,G)(y)\longrightarrow0
 \quad\text{for almost every }y.
 \label{eq2-add-geometric-smooth}
\end{equation}
For $0<\tau<1/2$, choose $\phi_\tau\in C_c^\infty((0,1))$ satisfying
$0\le\phi_\tau\le1$ and $\phi_\tau=1$ on $[\tau,1-\tau]$. Then:
\begin{align}
 |D_N(F,G)-D_N^{\phi_\tau}(F,G)|&\lesssim\tau+1/N,
 \label{eq2-add-time-cutoff-discrete}\\
 |C_N(F,G)-C_N^{\phi_\tau}(F,G)|&\lesssim\tau.
 \label{eq2-add-time-cutoff-continuous}
\end{align}
Take a countable sequence $\tau_l\downarrow0$, and intersect the co-null
sets from \eqref{eq2-add-geometric-smooth} for these $\phi_{\tau_l}$.
On their intersection,
\[
 \limsup_{j\to\infty}|D_{N_j}(F,G)-C_{N_j}(F,G)|\lesssim\tau_l
 \quad\text{for every }l.
\]
It follows that $D_{N_j}(F,G)-C_{N_j}(F,G)\to0$ there.
If $N_j\le N<N_{j+1}$, then
\[
 |D_N(F,G)-C_N(F,G)|
 \lesssim |D_{N_j}(F,G)-C_{N_j}(F,G)|+\frac{N-N_j}{N}.
\]
Since $N_{j+1}/N_j\to q$, the limit superior of the last fraction is at
most $q-1$. Now carry out the preceding argument for
$q=1+1/i$, $i\ge1$, and intersect the resulting countably many co-null
sets. On that intersection,
\[
 \limsup_{N\to\infty}|D_N(F,G)-C_N(F,G)|\lesssim1/i
 \quad\text{for every }i.
\]
Letting $i\to\infty$ proves \Cref{lem2-5}.
\end{proof}
\Cref{lem2-7} verifies the hypothesis of \Cref{lem2-5}
under the assumption of \Cref{thm3-1}.
\begin{lemma}\label[lemma]{lem2-7}
	Fix $F,G\in C^{\infty}_{1}({\bf Y})$ and $\phi \in C^{\infty}_{c}((0,\infty))$. Assuming \Cref{thm3-1}, there is some $\delta_c>0$ such that for each $N>1$, 
	$$\norm{D_{N}^{\phi}(F,G)-C_N^{\phi}(F,G)}_{L^{2}(\nu)}\lesssim_{F,G,\phi,c,\alpha,\beta} N^{-\delta_c}.$$
\end{lemma}
\begin{proof}[Proof of \Cref{lem2-7}]
The proof is divided into four steps.

\medskip
\noindent\textbf{Step 1. Derivation of \eqref{eq2-add-poisson-error}.}
For $N\ge2$ and $y\in Y$, set
\[
 q_{N,y}(t)=\phi(t/N)F(S^{\alpha t^c}y)G(S^{\beta t^c}y),
 \qquad t>0,
\]
and set $q_{N,y}(t)=0$ for $t\le0$. Then
$q_{N,y}\in C_c^\infty(\R)$, its support is contained in some $[aN,bN]\subset (0,\infty)$,
and extension by zero introduces no boundary singularity.
The Leibniz rule and the support length $O(N)$ of $q_{N,y}(t)$ now give
\begin{equation}
 \sup_{y\in Y}\norm{q_{N,y}^{(r)}}_{L^1(\R)}
 \lesssim_{r,F,G,\phi,c,\alpha,\beta}N^{1+r(c-1)}.
 \label{eq2-add-orbital-derivatives}
\end{equation}

Write
$I_{N,m}(y)=N^{-1}\widehat q_{N,y}(m)$ for all $m\in \Z$. Note that
\begin{align*}
 \frac1N\int_\R|q_{N,y}(t)e(-mt)|\,dt
 \le\frac1N\int_{aN}^{bN}|\phi(t/N)|\,dt
 =\int_a^b|\phi(u)|\,du<\infty.
\end{align*}
For $m\ne0$, integration by parts once gives
\begin{align*}
 \int_\R q_{N,y}(t)e(-mt)\,dt
 =-\frac1{2\pi im}\int_\R q_{N,y}(t)
       \frac{d}{dt}e(-mt)\,dt
 =\frac1{2\pi im}\int_\R q_{N,y}'(t)e(-mt)\,dt.
\end{align*}
Repeating this identity $r$ times gives
\[
 I_{N,m}(y)=\frac1{N(2\pi im)^r}
 \int_\R q_{N,y}^{(r)}(t)e(-mt)\,dt,
 \qquad m\ne0.
\]
Choose the constant $C_r$ in
\eqref{eq2-add-orbital-derivatives} so that its right-hand side is
$C_rN^{1+r(c-1)}$, and put $A_r=C_r/(2\pi)^r$. Then
\begin{equation}\label{eq2-0}
\sup_{y\in Y}|I_{N,m}(y)|
\le A_r N^{r(c-1)}|m|^{-r},\qquad m\ne0.
\end{equation}
For any fixed integer $r\ge2$,
\begin{align*}
 \sum_{m\ne0}\sup_{y\in Y}|I_{N,m}(y)|
 \le 2A_rN^{r(c-1)}\sum_{m=1}^\infty m^{-r}
 \le 2A_rN^{r(c-1)}\left(1+\frac1{r-1}\right)<\infty.
\end{align*}
Thus the series is absolutely and uniformly convergent in $y$ for
each fixed $N$. And
\[
 \sum_{m\ne0}\|I_{N,m}\|_{L^2(\nu)}
 \le\sum_{m\ne0}\sup_{y\in Y}|I_{N,m}(y)|<\infty.
\]
The series therefore also converges absolutely in $L^2(\nu)$.

For each $y$, $q_{N,y}$ is smooth and compactly supported,
so Poisson summation gives
\[
 \sum_{n\in\Z}q_{N,y}(n)=\sum_{m\in\Z}\widehat q_{N,y}(m).
\]
The left side is a finite sum. Since $q_{N,y}(n)=0$ for $n\le0$, its
quotient by $N$ is $D_N^\phi(F,G)(y)$. The term $m=0$ on the right,
divided by $N$, is $C_N^\phi(F,G)(y)$. Subtracting it yields
\begin{equation}
 D_N^\phi(F,G)-C_N^\phi(F,G)
 =\sum_{m\in\Z\setminus\{0\}}I_{N,m},\qquad
 I_{N,m}(y)=\frac1N\int_\R q_{N,y}(t)e(-mt)\,dt.
 \label{eq2-add-poisson-error}
\end{equation}
The preceding absolute-convergence estimates show that this is also an
identity in $L^2(\nu)$.

\medskip
\noindent\textbf{Step 2. Derivation of \eqref{eq2-add-low-frequency}.}
Set $\rho=1/c$. In the defining integral of $I_{N,m}$, make the
substitution
\[
 u=(t/N)^c,\qquad t=Nu^\rho,\qquad
 \frac{dt}{N}=\rho u^{\rho-1}\,du,\qquad t^c=N^cu.
\]
It follows that
\begin{align}
 I_{N,m}(y)
 &=\int_\R\eta(u)e(-mNu^\rho)
 F(S^{\alpha N^cu}y)G(S^{\beta N^cu}y)\,du,
 \label{eq2-add-poisson-mode}\\
 \eta(u)&=\rho u^{\rho-1}\phi(u^\rho).
 \notag
\end{align}
Here, $\eta$ is defined by the displayed formula for $u>0$ and extended
by zero to $u\le0$. Then we can say that $\eta\in C_c^\infty((0,\infty))$.

Choose a nonincreasing function $v\in C^\infty(\R)$ with $0\le v\le1$,
$v(s)=1$ for $s\le5/4$, and $v(s)=0$ for $s\ge7/4$. Put
\[
 \omega(s)=v(s)-v(2s).
\]
This is a real-valued smooth function. It vanishes for $s\le5/8$ and
for $s\ge7/4$, so its support is a compact subset of $(1/2,2)$.
For $J\ge0$ and a positive integer $m$, 
\begin{equation}
 v(m/2^J)=\sum_{j=0}^J\omega(m/2^j).
 \label{eq2-add-dyadic-identity}
\end{equation}

Choose the fixed parameters
\begin{equation}
 \theta=\frac c2,\qquad \sigma=\frac{1-\theta}{2}=\frac{2-c}{4},
 \qquad d_c=1+\left\lceil\frac2\sigma\right\rceil,
 \qquad P=2^{d_c}.
 \label{eq2-add-fixed-parameters}
\end{equation}
Since $1<c<2$, we have $0<\theta<1$ and $\sigma>0$. For each $N\ge2$,
put $J=\lfloor\theta\log_2N\rfloor$ and $P_0=2^J$. The inequalities
$J\le\theta\log_2N<J+1$ imply
$
N^\theta/2\le P_0\le N^\theta.
$
Hence, every $M=2^j$, $0\le j\le J$, satisfies
$1\le M\le N^\theta$, as required by \Cref{thm3-1}.

For $0\le j\le J$, define the two finite block sums
\[
 \mathcal B_{N,j}^{+}(y)=\sum_{\ell\ge1}\omega(\ell/2^j)I_{N,\ell}(y),
 \qquad
 \mathcal B_{N,j}^{-}(y)=\sum_{\ell\ge1}\omega(\ell/2^j)I_{N,-\ell}(y).
\]
By \eqref{eq2-add-poisson-mode}, the positive block is
\[
 \mathcal B_{N,j}^{+}(y)=
 \int_\R K_{N,2^j}(u)
 F(S^{\alpha N^cu}y)G(S^{\beta N^cu}y)\,du,
\]
where $K_{N,2^j}$ is the kernel of \Cref{thm3-1}, with the
fixed functions $\eta,\omega$ just chosen. \Cref{thm3-1},
with $R=N^c$
and $\|F\|_\infty,\|G\|_\infty\le1$, gives
\[
 \|\mathcal B_{N,j}^{+}\|_{L^2(\nu)}\lesssim_{c,\phi} N^{-\sigma/(2P)}.
\]
For the negative
block, use that $\omega$ is real to write
\begin{align*}
 \overline{\mathcal B_{N,j}^{-}(y)}
 &=\int_\R\overline{\eta(u)}
 \left(\sum_{\ell\in\Z}\omega(\ell/2^j)e(-\ell Nu^\rho)\right)
 \cdot\overline F(S^{\alpha N^cu}y)
 \overline G(S^{\beta N^cu}y)\,du.
\end{align*}
Apply \Cref{thm3-1} to $\overline\eta,\omega,\overline F,
\overline G$. It gives
\[
 \|\mathcal B_{N,j}^{-}\|_{L^2(\nu)}\lesssim_{c,\phi} N^{-\sigma/(2P)}.
\]

\eqref{eq2-add-dyadic-identity} now implies
\[
 \sum_{m\ne0}v(|m|/P_0)I_{N,m}
 =\sum_{j=0}^J\bigl(\mathcal B_{N,j}^{+}+\mathcal B_{N,j}^{-}\bigr).
\]
Put
$\kappa_c=\sigma/(4P)>0$. The triangle inequality yields
\begin{align*}
 \left\|\sum_{m\ne0}v(|m|/P_0)I_{N,m}\right\|_{L^2(\nu)}
 \le\sum_{j=0}^J\left(
 \|\mathcal B_{N,j}^{+}\|_{L^2(\nu)}
 +\|\mathcal B_{N,j}^{-}\|_{L^2(\nu)}\right)
 \lesssim_{c,\phi}(\log N)N^{-\sigma/(2P)}.
\end{align*}
Then we have 
\begin{align}
 \left\|\sum_{m\ne0}v(|m|/P_0)I_{N,m}\right\|_{L^2(\nu)}
 \lesssim_{c,\phi} N^{-\sigma/(4P)}.
 \label{eq2-add-low-frequency}
\end{align}

\medskip
\noindent\textbf{Step 3. Derivation of \eqref{eq2-add-high-frequency}.}
The remaining part of \eqref{eq2-add-poisson-error} is
\[
 \sum_{m\ne0}\bigl(1-v(|m|/P_0)\bigr)I_{N,m}.
\]
We have $0\le1-v\le1$, and this factor is zero whenever $|m|\le P_0$.
By the absolute $L^2$-convergence proved in {\bf Step 1} and \eqref{eq2-0}, for any fixed integer
$r\ge2$,
\begin{align*}
 &\left\|\sum_{m\ne0}\bigl(1-v(|m|/P_0)\bigr)I_{N,m}
              \right\|_{L^2(\nu)}
\le\sum_{|m|>P_0}\|I_{N,m}\|_{L^2(\nu)}
\le\sum_{|m|>P_0}\sup_{y\in Y}|I_{N,m}(y)|\\
 &\hspace{6cm}\le A_rN^{r(c-1)}\sum_{|m|>P_0}|m|^{-r}
 \le\frac{2A_r}{r-1}N^{r(c-1)}P_0^{1-r}.
\end{align*}

When $r>1$, $P_0\ge N^\theta/2$ gives
\[
 P_0^{1-r}\le(N^\theta/2)^{1-r}
 =2^{r-1}N^{\theta(1-r)}.
\]
Consequently, for any fixed integer
$r\ge2$,
\begin{align}
 \left\|\sum_{m\ne0}\bigl(1-v(|m|/P_0)\bigr)I_{N,m}
              \right\|_{L^2(\nu)}
 \lesssim_{F,G,\phi,c,\alpha,\beta,r} N^{\theta-r(\theta-c+1)}.
 \label{eq2-add-high-frequency}
\end{align}

\medskip
\noindent\textbf{Step 4. Fixing the derivative order and concluding.}
From \eqref{eq2-add-fixed-parameters},
\[
 \theta-c+1=1-\frac c2=\frac{2-c}{2}=2\sigma>0.
\]
Choose
\[
 r=\max\left\{2,
 \left\lceil\frac{\theta+\sigma/(4P)}{2\sigma}\right\rceil\right\}
\]
and fix it. Then $r$ depends only on $c$, and satisfies
\[
 \theta-r(\theta-c+1)=\theta-2r\sigma
 \le-\frac{\sigma}{4P}.
\]
\eqref{eq2-add-poisson-error},
\eqref{eq2-add-low-frequency}, and \eqref{eq2-add-high-frequency}
therefore give
\[
 \|D_N^\phi(F,G)-C_N^\phi(F,G)\|_{L^2(\nu)}
 \lesssim_{F,G,\phi,c,\alpha,\beta}N^{-\sigma/(4P)}.
\]
The required exponent is
\begin{equation}
 \delta_c=\frac{\sigma}{4P}
 =\frac{2-c}{16\,2^{d_c}}>0,
 \qquad d_c=1+\left\lceil\frac8{2-c}\right\rceil.
 \label{eq2-add-decay-exponent}
\end{equation}
This proves \Cref{lem2-7}.
\end{proof}
Finally, we can prove \Cref{TB}.
\begin{proof}
The conclusion follows from \Cref{lem2-7,lem2-5,lem2-6,lem2-3,lem2-1}.
\end{proof}
\section{Proof of \texorpdfstring{\Cref{thm3-1}}{Theorem~\getrefnumber{thm3-1}}}\label{sec3:main}
The estimate in \Cref{thm3-1} is the analytic input to
\Cref{lem2-7}. In this section, the notation $\|\cdot\|_2$ means the norm in $L^2(\nu)$ unless
another space is specified.
\begin{thm}\label{thm3-1}
Let $\rho,\theta\in(0,1)$, and fix
$\eta\in C_c^\infty((0,\infty);\C)$ and
$\omega\in C_c^\infty((1/2,2);\C)$. Set
\begin{equation}
 \sigma=\frac{1-\theta}{2},\qquad
 d_\theta=1+\left\lceil\frac2\sigma\right\rceil,
 \qquad P=2^{d_\theta}.
 \label{eq3-add-parameters}
\end{equation}
For an integer $N\ge2$ and a real number $M\in[1,N^\theta]$, define
$K_{N,M}:\R\to\C$ by
\begin{equation}
 K_{N,M}(u)=
 \begin{cases}
 \displaystyle\eta(u)\sum_{m\in\Z}\omega(m/M)e(-mNu^\rho),&u>0,\\[1mm]
 0,&u\le0.
 \end{cases}
 \label{eq3-add-kernel}
\end{equation}
For every measurable flow ${\bf Y}=(Y,\Y,\nu,(S^t)_{t\in\R})$,
every $\alpha,\beta\in\R$, every $R>0$, and all $1$-bounded
$F,G\in L^\infty(\nu)$,
\begin{equation}
 \left\|\int_\R K_{N,M}(u)
 F(S^{\alpha Ru})G(S^{\beta Ru})\,du\right\|_2^2
 \lesssim_{\rho,\theta,\eta,\omega}N^{-\sigma/P}.
 \label{eq3-add-main-bound}
\end{equation}
The implicit constant is independent of $N,M,R,\alpha,\beta,{\bf Y},F,G$.
\end{thm}

\noindent\textbf{Proof strategy for \Cref{thm3-1}.}
\Cref{prop3:scalar}, proved by Cauchy--Schwarz and the
spectral representation \cite[Theorem~4.45]{Folland2016}, reduces
\eqref{eq3-add-main-bound} to an integral of the supremum of the absolute value of the
Fourier transform of $K_{N,M}(u+h)\overline{K_{N,M}(u)}$.
We choose $H=N^{\sigma-1}$: the interval $0<h\le H$ satisfies
$NH\to\infty$ and $MH\to0$.
The second condition preserves the required second-derivative bounds
when the product of the two exponential sums is grouped by
$k=m-\ell$.

For $k\ne0$, \Cref{lem3-add-stationary} gives an exponential
sum over the integer index $j$. The derivative estimates in
\Cref{lem3-add-derivatives} and the counting estimates in
\Cref{lem3-add-resonance} give \Cref{prop3:offdiagonal}.

For $k=0$, inverse coordinates and \Cref{lem3-add-dual-formula}
produce a dual exponential sum. Nonvanishing derivatives of the inverse
phase allow \Cref{lem3-add-derivatives}(iii) to be applied, giving
\Cref{prop3:diagonal}. \Cref{sec3-13} integrates
\eqref{eq3-add-offdiagonal-goal}, \eqref{eq3-add-tiny-integral}, and
\eqref{eq3-add-diagonal-high-derivative} to obtain
\eqref{eq3-add-main-bound}. The auxiliary results are collected in
\Crefrange{app:derivatives}{app:resonance}.

\subsection{Reduction to scalar correlations}\label{sec3:scalar}

If either $\eta$ or $\omega$ is identically zero,
\eqref{eq3-add-main-bound} is immediate. Otherwise,
choose a closed interval $I_0=[a_0,b_0]\subset(0,\infty)$ with
$\operatorname{supp}\eta\subset(a_0,b_0)$.
The sum over $m$ in \eqref{eq3-add-kernel} is finite, and
$K_{N,M}\in C_c^\infty(\R)$.
For real $h,\xi$, define
\begin{equation}
 \mathcal A_h(\xi)=\int_\R
 K_{N,M}(u+h)\overline{K_{N,M}(u)}e(\xi u)\,du.
 \label{eq3-add-correlation}
\end{equation}
This integral is absolutely convergent. Dominated convergence, locally
in $(h,\xi)$, proves continuity, while Cauchy--Schwarz gives
$$|\mathcal A_h(\xi)|\le\|K_{N,M}\|_{L^2(\R)}^2.$$

\begin{prop}\label[prop]{prop3:scalar}
Under the assumptions of \Cref{thm3-1}, for every real
$H\in(0,1]$,
\begin{equation}
 \left\|\int_\R K_{N,M}(u)
 F(S^{\alpha Ru})G(S^{\beta Ru})\,du\right\|_2^2
 \lesssim_{|I_0|}\frac1H\int_0^H
                \sup_{\xi\in\R}|\mathcal A_h(\xi)|\,dh.
 \label{eq3-add-spectral-reduction}
\end{equation}
\end{prop}
\begin{proof}
Choose measurable representatives of $F,G$ bounded by one everywhere,
and put
\[
 V(u)(y)=K_{N,M}(u)F(S^{\alpha Ru}y)G(S^{\beta Ru}y).
\]
Joint measurability of ${\bf Y}$, together with separability of $L^2(\nu)$, implies that $u\mapsto V(u)$ is strongly measurable. Since $$\|V(u)\|_2\le|K_{N,M}(u)|$$ and $K_{N,M}\in C_c^\infty(\R)$,\[
 V\in L^1(\R;L^2(\nu))\cap L^2(\R;L^2(\nu)).
\]

Define
$$V_H(u)=H^{-1}\int_0^H V(u+s)\,ds$$ for almost every $u\in\R$. Then
$$\int_\R V_H(u)\,du=\frac1H\int_0^H\int_\R V(u+s)\,du\,ds=\int_\R V(u)\,du.$$
The support of $V_H$ is contained in $J_H=[a_0-H,b_0]$, whose length is
$$|J_H|=|I_0|+H\le |I_0|+1.$$ 
Note that 
\begin{align*}
& \left\|\int_\R V\right\|_2=\sup_{\substack{W\in L^2(\nu)\\ \|W\|_2=1}} \left|\left\langle\int_{J_H}V_H(u)\,du,W\right\rangle\right|
\\ & \hspace{3cm}\le\sup_{\substack{W\in L^2(\nu)\\ \|W\|_2=1}}
\int_{J_H}|\langle V_H(u),W\rangle|\,du
\le\int_{J_H}\|V_H(u)\|_2\,du.
\end{align*}
Therefore,
\begin{align*}
 &\left\|\int_\R V\right\|_2^2
 \le \left(\int_{J_H}\|V_H(u)\|_2\,du\right)^2
 \le (|I_0|+1)\int_\R\|V_H(u)\|_2^2\,du\\
 &\hspace{3cm}=\frac{|I_0|+1}{H^2}\int_0^H\int_0^H\int_\R
       \langle V(u+s),V(u+s')\rangle\,du\,ds\,ds'.
\end{align*}

Translating $u$ by $s'$ in the innermost integral gives
\[
 \int_\R\langle V(u+s),V(u+s')\rangle\,du=C(s-s'),
 \qquad
 C(h):=\int_\R\langle V(u+h),V(u)\rangle\,du.
\]
For every integrable scalar function $\Phi$ on $[-H,H]$,
\[
 \int_0^H\int_0^H\Phi(s-s')\,ds\,ds'
 =\int_{-H}^H(H-|h|)\Phi(h)\,dh.
\]
Indeed, set $h=s-s'$ and $t=s'$. The square $[0,H]^2$ becomes
\[
 \{(h,t):-H\le h\le H,\ 0\le t\le H,\ 0\le t+h\le H\}.
\]
For fixed $h$, the allowed $t$-interval is
$[\max(0,-h),\min(H,H-h)]$, whose length is $H-|h|$. Applying the above arguments with $\Phi=C$ and using $H-|h|\le H$ gives
\[
 \left\|\int_\R V\right\|_2^2
 \le\frac{|I_0|+1}{H}\int_{-H}^H|C(h)|\,dh.
\]

For a fixed $h\in\R$, let
$F_h=(F\circ S^{\alpha Rh})\overline F$ and
$G_h=(G\circ S^{\beta Rh})\overline G$.
Changing variables by $z=S^{\alpha Ru}y$ in the inner product gives
\begin{align*}
 \langle V(u+h),V(u)\rangle
 ={}&K_{N,M}(u+h)\overline{K_{N,M}(u)}
 \int_Y F_h(z)G_h(S^{(\beta-\alpha)Ru}z)\,d\nu(z).
\end{align*}
Set $U_uW=W\circ S^{(\beta-\alpha)Ru}$ on $L^2(\nu)$.
The last integral is $\langle U_uG_h,\overline{F_h}\rangle$.

We need to verify strong continuity of the action $U$ before applying the spectral theorem. Joint measurability of ${\bf Y}$ implies that
$t\mapsto\langle U_tW,Z\rangle$ is measurable for every $W,Z\in L^2(\nu)$. For
$W\in L^2(\nu)$ and $a\in C_c^\infty(\R)$, define 
\[
 W_a=\int_\R a(t)U_tW\,dt.
\]
The change of variables $r=t+s$ gives
\[
 U_sW_a=\int_\R a(r-s)U_rW\,dr,
\]
and hence
\[
 \|U_sW_a-W_a\|_2
 \le \|a(\cdot-s)-a\|_{L^{1}(\R)}\|W\|_2\longrightarrow0
 \qquad(s\to0).
\]
It remains to show that the span of the vectors $W_a$ is dense in $L^2(\nu)$. Suppose that
$Z$ is orthogonal to every such vector. Then
\[
 0=\langle W_a,Z\rangle
 =\int_\R a(t)\langle U_tW,Z\rangle\,dt
\]
for every $a\in C_c^\infty(\R)$. Thus
$t\mapsto\langle U_tW,Z\rangle$ vanishes almost everywhere for every $W$.
Choose a countable dense subset $\mathcal D\subset L^2(\nu)$ and a single
$t_0$ outside the union of the corresponding null sets. Then
$\langle U_{t_0}W,Z\rangle=0$ for every $W\in\mathcal D$, hence for every
$W\in L^2(\nu)$. Since $U_{t_0}$ is unitary, $Z=0$. Therefore, the span of the vectors $W_a$ is dense in $L^2(\nu)$. This proves the strong continuity of the action $U$.

By the spectral theorem (see \cite[Theorem~4.45]{Folland2016}), there is a
projection-valued measure $E$ on $\R$ such that
$$U_u=\int_\R e(u\xi)\,dE(\xi).$$ For $\varkappa\in L^1(\R)$, we have
\[
 \int_\R\varkappa(u)U_u\,du
 =\int_\R\left(\int_\R\varkappa(u)e(u\xi)\,du\right)dE(\xi).
\]
The operator on the left is defined strongly on each vector. The
spectral functional calculus bounds its norm by the supremum of the
scalar function in parentheses. Apply this to
$\varkappa(u)=K_{N,M}(u+h)\overline{K_{N,M}(u)}$ to obtain
\[
 \left|\int_\R\langle V(u+h),V(u)\rangle\,du\right|
 \le\sup_{\xi\in\R}|\mathcal A_h(\xi)|.
\]
Finally,
$\mathcal A_{-h}(\xi)=e(\xi h)\overline{\mathcal A_h(-\xi)}$,
so the two half-intervals in $h$ give equal suprema.
This proves \eqref{eq3-add-spectral-reduction}.
\end{proof}

We now fix the real scales
\begin{equation}
 L=N^\sigma,\qquad H=L/N=N^{\sigma-1},\qquad \tau=Nh.
 \label{eq3-add-shift-scales}
\end{equation}
The substitution $h=\tau/N$ in \Cref{prop3:scalar} gives
\begin{equation}
 MH\le N^{-\sigma},\qquad
 \left\|\int V\right\|_2^2
 \lesssim_{|I_0|}\frac1L\int_0^L
       \sup_{\xi\in\R}|\mathcal A_{\tau/N}(\xi)|\,d\tau.
 \label{eq3-add-dual-time-reduction}
\end{equation}
Here $H\to0$, $L\to\infty$, and $MH\to0$ uniformly in the allowed
$M$. These three facts are used separately below.

Let $\psi:(0,\infty)\to\R$ be $\psi(u)=u^\rho$.
For $h\ge0$, put $w_h(u)=\eta(u+h)\overline{\eta(u)}$ on $\R$,
and for $u>0$ define
\begin{equation}
 \Delta_h\psi(u)=\psi(u+h)-\psi(u),\qquad
 \psi_h(u)=\int_0^1\psi'(u+sh)\,ds.
 \label{eq3-add-divided-difference}
\end{equation}
Thus $\Delta_h\psi=h\psi_h$ and $\psi_0=\psi'$.
On every fixed compact subinterval of $(0,\infty)$, all fixed-order
$u$-derivatives of $\Delta_h\psi$ are $O(h)$ for $0\le h\le1$.
Also $\operatorname{supp}w_h\subset\operatorname{supp}\eta$, and
$\|w_h\|_{C^s(\R)}$ is bounded uniformly in $h$ for each fixed $s$.

For each $k\in\Z$, define a smooth function $a_k:\R\to\C$ by
\[
 a_k(j)=\omega(j/M)\overline{\omega((j-k)/M)}\qquad(j\in\R).
\]
For real $j,\xi$ and $0<h\le H$, let
\begin{equation}
 J_{k,j}(\xi,h)=\int_{I_0}w_h(u)
 e\bigl(\xi u-Nk\psi(u)-Nj\Delta_h\psi(u)\bigr)\,du.
 \label{eq3-add-J}
\end{equation}

By the definitions of $K_{N,M}, \mathcal A_h(\xi)$ and \eqref{eq3-add-J}, 
\begin{equation}
 \begin{gathered}
 \mathcal A_h(\xi)=\sum_{k\in\Z}\sum_{j\in\mathcal J_k}
          a_k(j)J_{k,j}(\xi,h),\\
 \mathcal J_k=\{j\in\Z:M/2<j<2M,\ M/2<j-k<2M\}.
 \end{gathered}
 \label{eq3-add-mode-expansion}
\end{equation}
Only $|k|\le2M$ contribute, and $|\mathcal J_k|\le3M$.
Each $\mathcal J_k$ is empty or consists of consecutive integers.
Define
\[
 \mathcal D_h(\xi)=\sum_{j\in\mathcal J_0}a_0(j)J_{0,j}(\xi,h),\qquad
 \mathcal O_h(\xi)=\sum_{\substack{k\in\Z\\0<|k|\le2M}}
                       \sum_{j\in\mathcal J_k}a_k(j)J_{k,j}(\xi,h).
\]
Thus $\mathcal A_h=\mathcal D_h+\mathcal O_h$, where $\mathcal D_h$ is the $k=0$ contribution and $\mathcal O_h$ contains the terms with $k\ne0$. 

\subsection{The off-diagonal contribution}\label{sec3:offdiagonal}
In this subsection, we build the following proposition, which includes two estimates of $\mathcal O_h$.
\begin{prop}\label[prop]{prop3:offdiagonal}
There exist $\tau_*\in(0,1]$ and an integer $N_{\rm off}\ge2$,
depending only on $\mathfrak d:=(\rho,\theta,\eta,\omega)$, such that, for every integer
$N\ge N_{\rm off}$, real $M\in[1,N^\theta]$, and real $h\in(0,H]$,
\begin{align}
 \sup_{\xi\in\R}|\mathcal O_h(\xi)|
 &\lesssim_{\mathfrak d,\epsilon} N^\epsilon
       \left(Mh+\frac1{Nh}+\sqrt{\frac MN}\right)
       &&(\epsilon>0,\tau_*\le Nh\le L),
       \label{eq3-add-offdiagonal-goal}\\
 \sup_{\xi\in\R}|\mathcal O_h(\xi)|
 &\lesssim_{\mathfrak d}
       \sqrt{\frac MN}\min\{M,(Nh)^{-1}\}+MN^{-3/2}
       &&(0<Nh<\tau_*).
       \label{eq3-add-tiny-offdiagonal}
\end{align}
\end{prop}
\begin{proof}
The proof is divided into six steps.

\medskip
\noindent\textbf{Step 1. Fixed intervals and a uniform perturbation bound.}
Set
\[
 r_-:=\tfrac12\psi'(b_0),\qquad
 r_+:=2\psi'(a_0),\qquad J_\psi:=[r_-,r_+].
\]
Because $\psi'$ decreases from $+\infty$ to zero, choose
$0<a_1<a_0<b_0<b_1$ with
$\psi'(a_1)>8r_+$ and $\psi'(b_1)<r_-/8$.
Fix these endpoints using only $\rho,I_0$, and put
\[
 I_1=[a_1,b_1],\qquad I_2=[a_1/2,b_1+1].
\]
Thus $I_1+[0,1]\subset I_2$. On $I_2$, the functions
$\psi',-\psi'',\psi^{(3)}$ have positive minima. In particular,
\[
 \begin{gathered}
 d_0:=\min\{r_-,r_+/2\}>0,\qquad
 m_0:=\min_{I_1}|\psi''|>0,\\
 g_0:=\min\{\psi'(a_1)-r_+,r_- -\psi'(b_1)\}>0.
 \end{gathered}
\]
For each integer $s\ge0$, the identity
\eqref{eq3-add-divided-difference}, differentiated in $u$, gives
\[
 \|\Delta_h\psi\|_{C^s(I_1)}
 \le h\max_{\ell\in\{0,\ldots,s\}}\sup_{u\in I_2}|\psi^{(\ell+1)}(u)|
                  \qquad(0\le h\le 1).
\]
Then for $|k|\ge1,j\in \mathcal O_M:=(-1,2M+1),h\in (0,H]$,
\begin{equation}
 \left\|\frac jk\Delta_h\psi\right\|_{C^s(I_1)}
 \le C_{\mathfrak d,s}\frac{Mh}{|k|}
 \le C_{\mathfrak d,s}N^{-\sigma}\qquad (1\le M\le N^{\theta}).
 \label{eq3-add-phase-perturbation}
\end{equation}
Choose one constant $C_{\mathfrak d,8}^*$ valid simultaneously for
$0\le s\le8$, and let
\[
 \delta_0=\frac14\min\{1,d_0,m_0,g_0\},\qquad
 N_{\rm off}=\max\left\{2,
 \left\lceil(C_{\mathfrak d,8}^*/\delta_0)^{1/\sigma}\right\rceil
 \right\}.
\]
Increasing this integer, if necessary, does not change its permitted
dependencies. 

For the rest of this proof, fix $N,M,h,\xi$ in the stated ranges.

\medskip
\noindent\textbf{Step 2. Stationary and nonstationary indices.}
For each nonzero integer $k$ with $|k|\le2M$, set
$r_k=\xi/(Nk)\in\R$. For $j\in\mathcal O_M$, define
$p_{k,j}:(0,\infty)\to\R$ by
\begin{equation}
 p_{k,j}(u)=\operatorname{sgn}(k)
 \left(r_k u-\psi(u)-\frac jk\Delta_h\psi(u)\right).
 \label{eq3-add-normalized-phase}
\end{equation}
This is the phase in \eqref{eq3-add-J} divided by $N|k|$.
If $r_k\notin J_\psi$, then
$\operatorname{dist}(r_k,\psi'(I_0))\ge d_0$.
\eqref{eq3-add-phase-perturbation} therefore gives
$|p_{k,j}'|\ge d_0/2$ on $I_0$.
Every derivative of $p_{k,j}$ of order at least two is bounded on $I_1$
independently of $r_k$, since the term $r_ku$ has no such derivative.
For each fixed integer $1\le B\le 7$, the nonstationary estimate
\eqref{eq3-add-nonstationary} yields
\begin{equation}
 J_{k,j}(\xi,h)=O_{\mathfrak d,B}((N|k|)^{-B})
                   \qquad(r_k\notin J_\psi).
 \label{eq3-add-nonstationary-modes}
\end{equation}
Taking $B=2$, using $|a_k(j)|\le\|\omega\|_\infty^2$ and
$|\mathcal J_k|\le3M$, gives
\begin{equation}
 \sum_{\substack{k\in\Z,\ 0<|k|\le2M\\r_k\notin J_\psi}}
 \sum_{j\in\mathcal J_k}|a_k(j)J_{k,j}(\xi,h)|
 \lesssim_{\mathfrak d}MN^{-2}.
 \label{eq3-add-nonstationary-total}
\end{equation}

Suppose $r_k\in J_\psi$. The function
\[
 \Phi_{k,j}(u)=\psi'(u)+(j/k)(\Delta_h\psi)'(u),\qquad (u>0,j\in \mathcal O_M)
\]
is strictly decreasing on $I_1$: its derivative is at most $-3m_0/4$.
Its value at $a_1$ is at least $r_++3g_0/4$, and its value at $b_1$
is at most $r_- -3g_0/4$. Thus there is a unique
$u_k(j)\in(a_1,b_1)$ satisfying
\begin{equation}
 k\psi'(u_k(j))+j(\Delta_h\psi)'(u_k(j))=\xi/N.
 \label{eq3-add-critical-equation}
\end{equation}
On $I_1$, the derivative upper bound
$|\Phi_{k,j}'|\le B_1:=1+\max_{I_1}|\psi''|$ follows from
$\delta_0\le1/4$. The mean value theorem implies
\[
 u_k(j)-a_1\ge\frac{3g_0}{4B_1},\qquad
 b_1-u_k(j)\ge\frac{3g_0}{4B_1}.
\]
Consequently all these points lie in
\[
 I_*=[a_1+\delta_*,b_1-\delta_*],\qquad
 \delta_*:=\min\{(b_1-a_1)/4,g_0/(4B_1)\}>0.
\]
Note that for any $j\in \mathcal O_M$, we have $$\psi'(u_k(j))+(j/k)(\Delta_h\psi)'(u_k(j))-r_k=0.$$ Therefore, the implicit function theorem gives a smooth function
$u_k:\mathcal O_M\to(a_1,b_1)$; all its values lie in the fixed
interval $I_*$. Define, for real $j\in\mathcal O_M$,
\begin{equation}
 \begin{gathered}
 D_k(j)=k\psi''(u_k(j))+j(\Delta_h\psi)''(u_k(j))=k\Phi_{k,j}'(u_k(j)),\\
 |D_k(j)|\asymp_{\mathfrak d}|k|,\qquad
 \operatorname{sgn}D_k(j)=-\operatorname{sgn}(k).
 \end{gathered}
 \label{eq3-add-critical-curvature}
\end{equation}
The above estimate follows from that on $I_1$, $\Phi_{k,j}'(u_k(j))\le -3m_0/4<0$ and
$$3m_0|k|/4\le |D_k(j)|=|k||\Phi_{k,j}'|\le B_1|k|.$$ 

Let
\[
 \mathcal K(\xi)=\{k\in\Z:0<|k|\le2M,\ r_k\in J_\psi\}.
\]
If $\mathcal K(\xi)$ is empty,
\eqref{eq3-add-nonstationary-total} already proves both required
estimates: the quantity $MN^{-2}$ is bounded by both $MN^{-3/2}$
and $\sqrt{M/N}$. Otherwise, define $Q_\xi=|\xi|/N>0$. For $k\in\mathcal K(\xi)$,
\begin{equation}
 \operatorname{sgn}k=\operatorname{sgn}\xi,\qquad
 \frac{Q_\xi}{r_+}\le|k|\le\frac{Q_\xi}{r_-},\qquad
 r_-\le Q_\xi\le2r_+M.
 \label{eq3-add-stationary-k-range}
\end{equation}
Hence,
\begin{equation}
 |\mathcal K(\xi)|\lesssim_{\mathfrak d}Q_\xi,\qquad
 \sum_{k\in\mathcal K(\xi)}|k|^{-1}\lesssim_{\mathfrak d}1,
 \qquad |k|\asymp_{\mathfrak d}Q_\xi.
 \label{eq3-add-stationary-count}
\end{equation}
Indeed, the interval for $|k|$ has length
$Q_\xi(r_-^{-1}-r_+^{-1})$, so its number of integers is at most
that length plus one. Each reciprocal is at most $r_+/Q_\xi$;
using $Q_\xi\ge r_-$ gives the second estimate. The sign condition
in \eqref{eq3-add-stationary-k-range} makes $k\mapsto|k|$ injective.

\medskip
\noindent\textbf{Step 3. Stationary phase and the sum over $j$.}
For $k\in\mathcal K(\xi)$ and $j\in\mathcal O_M$, set
\begin{equation}
 \Theta_k(j)=\xi u_k(j)-Nk\psi(u_k(j))
                    -Nj\Delta_h\psi(u_k(j)).
 \label{eq3-add-action}
\end{equation}
The normalized phases $p_{k,j}$ have uniformly bounded $C^8(I_1)$-norms because $r_k\in J_\psi$. Their second derivatives have sign
$\operatorname{sgn}(k)$ and absolute value at least $3m_0/4$ on
$I_1$. The amplitude $w_h$ has support in $I_0$
and uniformly bounded $C^6(\R)$-norm. \Cref{lem3-add-stationary}
therefore applies with phase parameter $N|k|$ and gives
\begin{equation}
 \begin{split}
 a_k(j)J_{k,j}(\xi,h)
 &=(N|k|)^{-1/2}b_k(j)e(\Theta_k(j))+\mathcal R_{k,j},\\
 |\mathcal R_{k,j}|&\le C_{\mathfrak d}(N|k|)^{-3/2}
                         \qquad(j\in\mathcal J_k),
 \end{split}
 \label{eq3-add-mode-stationary}
\end{equation}
where
\begin{equation}
 b_k(j)=e(\operatorname{sgn}(k)/8)
 \frac{a_k(j)w_h(u_k(j))}
 {\bigl|\psi''(u_k(j))+(j/k)(\Delta_h\psi)''(u_k(j))\bigr|^{1/2}}.
 \label{eq3-add-stationary-amplitude}
\end{equation}
The function $b_k$ is smooth on $\mathcal O_M$. 

Differentiating \eqref{eq3-add-critical-equation} in the real variable
$j$, with $N,M,h,\xi,k$ fixed, gives
\begin{equation}
 u_k'(j)=-\frac{(\Delta_h\psi)'(u_k(j))}{D_k(j)}
                =O_{\mathfrak d}(h/|k|).
 \label{eq3-add-critical-point-derivative}
\end{equation}
To differentiate the denominator in
\eqref{eq3-add-stationary-amplitude}, set $E_k(j)=D_k(j)/k<0$.
The chain rule gives
\[
 E_k'(j)=\left(\psi^{(3)}(u_k(j))
       +\frac jk(\Delta_h\psi)^{(3)}(u_k(j))\right)u_k'(j)
       +\frac1k(\Delta_h\psi)''(u_k(j)).
\]
The parenthesis is bounded by some $C_{\mathfrak d}$; thus
$|E_k'|\le C_{\mathfrak d}h/|k|$. The other factors satisfy
\begin{align*}
 a_k'(j)&=M^{-1}\omega'(j/M)\overline{\omega((j-k)/M)}
   +M^{-1}\omega(j/M)\overline{\omega'((j-k)/M)},\\
 \frac d{dj}w_h(u_k(j))&=w_h'(u_k(j))u_k'(j),\\
 \frac d{dj}(-E_k(j))^{-1/2}
   &=\frac12(-E_k(j))^{-3/2}E_k'(j).
\end{align*}
In particular, $|a_k'|\le2\|\omega\|_\infty\|\omega'\|_\infty/M$.
Using $-E_k\ge3m_0/4$ in the product rule proves
\begin{equation}
 \begin{gathered}
 \|b_k\|_{L^\infty([0,2M])}\le C_{\mathfrak d},\qquad
 |b_k'(j)|\le C_{\mathfrak d}(M^{-1}+h/|k|),\\
 \mathcal V_{[0,2M]}(b_k)
 \le C_{\mathfrak d}\left(1+\int_0^{2M}(M^{-1}+h/|k|)\,dj\right)
 \lesssim_{\mathfrak d}1+Mh/|k|\lesssim_{\mathfrak d}1.
 \end{gathered}
 \label{eq3-add-amplitude-variation}
\end{equation}

Define
\[
 \begin{aligned}
 \mathcal S_k(\xi,h)&=(N|k|)^{-1/2}
             \sum_{j\in\mathcal J_k}b_k(j)e(\Theta_k(j)),\\
 \mathcal E_{\rm ns}(\xi,h)&=
 \sum_{\substack{k\in\Z,\ 0<|k|\le2M\\r_k\notin J_\psi}}
 \sum_{j\in\mathcal J_k}a_k(j)J_{k,j}(\xi,h).
 \end{aligned}
\]
An empty $\mathcal J_k$ contributes zero. Let
$$\mathcal E_{\rm st}(\xi,h)=\sum_{k\in\mathcal K(\xi)}\sum_{j\in\mathcal J_k}\mathcal R_{k,j}.$$ From
\eqref{eq3-add-mode-stationary} and
\eqref{eq3-add-nonstationary-total},
\begin{equation}
 \begin{split}
 \mathcal O_h(\xi)&=\sum_{k\in\mathcal K(\xi)}\mathcal S_k(\xi,h)
                 +\mathcal E_{\rm st}(\xi,h)+\mathcal E_{\rm ns}(\xi,h),\\
 |\mathcal E_{\rm st}(\xi,h)|
 &\le\sum_{k\in\mathcal K(\xi)}\sum_{j\in\mathcal J_k}
                  |\mathcal R_{k,j}|\lesssim_{\mathfrak d}MN^{-3/2},\\
 |\mathcal E_{\rm ns}(\xi,h)|&\lesssim_{\mathfrak d}MN^{-2}.
 \end{split}
 \label{eq3-add-remainder-sum}
\end{equation}
For the stationary bound we used
$|\mathcal J_k|\le3M$ and
$$\sum_{k\ne0}|k|^{-3/2}=2\sum_{p\ge1}p^{-3/2}<\infty.$$

Differentiation of \eqref{eq3-add-action} yields
\[
 \Theta_k'(j)=-N\Delta_h\psi(u_k(j))
 +\bigl(\xi-Nk\psi'(u_k(j))-Nj(\Delta_h\psi)'(u_k(j))\bigr)u_k'(j).
\]
The parenthesis is zero by \eqref{eq3-add-critical-equation}.
Differentiating again and using
\eqref{eq3-add-critical-point-derivative}, we obtain
\begin{align}
 \Theta_k'(j)&=-N\Delta_h\psi(u_k(j)),
       \label{eq3-add-action-first}\\
 \Theta_k''(j)&=N\frac{((\Delta_h\psi)'(u_k(j)))^2}{D_k(j)}.
       \label{eq3-add-action-second}
\end{align}
Since $u_k(j)+[0,h]\subset I_2$ and $\psi''$ has one negative sign,
\[
 h\min_{I_2}|\psi''|
 \le\left|\int_0^h\psi''(u_k(j)+s)\,ds\right|
 =|(\Delta_h\psi)'(u_k(j))|
 \le h\max_{I_2}|\psi''|.
\]
Define the positive scale $\Lambda_\xi=Nh^2/Q_\xi$. Then
\begin{equation}
 |\Theta_k''(j)|\asymp_{\mathfrak d}\Lambda_\xi,
 \qquad \operatorname{sgn}\Theta_k''(j)=-\operatorname{sgn}k
                         \quad(0\le j\le2M).
 \label{eq3-add-mode-curvature}
\end{equation}
\Cref{lem3-add-derivatives}(ii), \eqref{eq3-add-amplitude-variation}, and Abel summation
therefore give
\[
 \left|\sum_{j\in\mathcal J_k}b_k(j)e(\Theta_k(j))\right|
 \lesssim_{\mathfrak d}M\Lambda_\xi^{1/2}+\Lambda_\xi^{-1/2}.
\]
Multiplication by $(N|k|)^{-1/2}$ and
$|k|\asymp_{\mathfrak d}Q_\xi$ give
\begin{equation}
 \begin{split}
 |\mathcal S_k(\xi,h)|&\lesssim_{\mathfrak d}
                    \frac{Mh}{|k|}+\frac1{Nh},\\
 \left|\sum_{j\in\mathcal J_k}a_k(j)J_{k,j}(\xi,h)\right|
 &\lesssim_{\mathfrak d}
       \frac{Mh}{|k|}+\frac1{Nh}+M(N|k|)^{-3/2}.
 \end{split}
 \label{eq3-add-mode-second-bound}
\end{equation}

For the error $\Theta_k'(j)+{h\xi}/{k}$, the exact identities
\begin{align}
 h\psi'(u)-\Delta_h\psi(u)
     &=-\int_0^h(h-s)\psi''(u+s)\,ds,
          \label{eq3-add-divided-identities}\\
 (\Delta_h\psi)'(u)&=\int_0^h\psi''(u+s)\,ds \nonumber
\end{align}
imply bounds $C_{\mathfrak d}h^2$ and $C_{\mathfrak d}h$, respectively,
on $I_1$. Multiplying \eqref{eq3-add-critical-equation} by $Nh/k$
and using \eqref{eq3-add-action-first}, we obtain
\[
 \Theta_k'(j)+\frac{h\xi}{k}
 =N\left[h\psi'(u_k(j))-\Delta_h\psi(u_k(j))
       +\frac jk h(\Delta_h\psi)'(u_k(j))\right].
\]
Since $|j/k|\le C_{\mathfrak d}M/Q_\xi$ and
$Q_\xi\le2r_+M$, choose $C_U=C_U(\mathfrak d)\ge1$ such that
\begin{equation}
 U_\xi=C_U M\Lambda_\xi=C_U\frac{NMh^2}{Q_\xi},\qquad
 \left|\Theta_k'(j)+\frac{h\xi}{k}\right|\le U_\xi
                      \quad(0\le j\le2M).
 \label{eq3-add-derivative-window}
\end{equation}

Choose fixed $0<a_\psi\le b_\psi$ such that
$a_\psi\le\psi_h(u)\le b_\psi$ on $I_1$ for $0\le h\le1$.
Set
\[
 \tau_*:=\min\{1,(4b_\psi)^{-1}\},\qquad
 u_0:=\min\{2^{-12},C_U2^{-20}\}.
\]
These constants depend only on $\mathfrak d$, and
$8u_0+8\sqrt{u_0/C_U}<1/8$.

\medskip
\noindent\textbf{Step 4. The range $\tau=Nh\ge\tau_*$ with $U_\xi\ge u_0$.}
Using \eqref{eq3-add-stationary-count} in the first line of
\eqref{eq3-add-mode-second-bound}, we get
\[
 \sum_{k\in\mathcal K(\xi)}|\mathcal S_k(\xi,h)|
 \lesssim_{\mathfrak d}Mh+\frac{Q_\xi}{Nh}.
\]
The assumption $U_\xi\ge u_0$ implies
$Q_\xi\le(C_U/u_0)NMh^2$, so the second term is at most
$(C_U/u_0)Mh$. Moreover
\[
 MN^{-3/2}\le\sqrt{M/N},\qquad MN^{-2}\le\sqrt{M/N}
                                  \quad(1\le M\le N^{\theta}).
\]
Adding \eqref{eq3-add-remainder-sum} proves
\eqref{eq3-add-offdiagonal-goal} in this case, with no $N^\epsilon$
loss needed.

\medskip
\noindent\textbf{Step 5. The range $\tau=Nh\ge\tau_*$ with $0<U_\xi<u_0$.}
Put
\[
 A_\xi=h|\xi|=\tau Q_\xi,\qquad
 \vartheta_\xi=8U_\xi+8\sqrt{\Lambda_\xi}.
\]
Then $U_\xi\le\vartheta_\xi/8$ and
$0<\vartheta_\xi<1/8$, since
$\Lambda_\xi=U_\xi/(C_UM)$ and $M\ge1$.
For positive integers $p$, define
$d_\xi(p)=\|A_\xi/p\|_{\R/\Z}$.
We apply \Cref{lem3-add-resonance} on
$[Q_\xi/r_+,Q_\xi/r_-]\cap\N$, with
$c_1=1/r_+$, $c_2=1/r_-$, $q_*=r_-$, and $B=1$.
Since $Q_\xi\ge r_-$, $\tau\ge\tau_*$, and
\begin{equation}
 A_\xi+Q_\xi\le2r_+(L+1)M
       \le4r_+N^{\sigma+\theta}\le C_*N,
 \qquad C_*:=\max\{1,4r_+\},
 \label{eq3-add-resonance-size}
\end{equation}
the assumptions of \Cref{lem3-add-resonance} are satisfied.

Let
\[
 \mathcal K_{\rm res}(\xi,h)
   =\{k\in\mathcal K(\xi):d_\xi(|k|)\le2\vartheta_\xi\},
 \qquad
 \mathcal K_{\rm nr}(\xi,h)=\mathcal K(\xi)\setminus\mathcal K_{\rm res}(\xi,h).
\]
Using the threshold $2\vartheta_\xi<1/4$ and the injectivity of
$k\mapsto|k|$, \Cref{lem3-add-resonance} gives, for every
$\epsilon'>0$,
\begin{equation}
 |\mathcal K_{\rm res}(\xi,h)|
 \lesssim_{\mathfrak d,\epsilon'}N^{\epsilon'}
                                      (1+Q_\xi\vartheta_\xi).
 \label{eq3-add-resonant-count-used}
\end{equation}
The first line of \eqref{eq3-add-mode-second-bound} now yields
\begin{equation}
 \begin{split}
 \left|\sum_{k\in\mathcal K_{\rm res}(\xi,h)}\mathcal S_k(\xi,h)\right|
 &\lesssim_{\mathfrak d,\epsilon'}N^{\epsilon'}
        (1+Q_\xi\vartheta_\xi)
        \left(\frac{Mh}{Q_\xi}+\frac1{Nh}\right)\\
 &\lesssim_{\mathfrak d,\epsilon'}N^{\epsilon'}
        \left(Mh+\frac1{Nh}+\sqrt{Q_\xi/N}\right).
 \end{split}
 \label{eq3-add-resonant-total}
\end{equation}
To obtain the second line, expand the product into
\[
 \frac{Mh}{Q_\xi}+\frac1{Nh}+Mh\vartheta_\xi
                            +\frac{Q_\xi\vartheta_\xi}{Nh}.
\]
The first and third terms are bounded by $C_{\mathfrak d}Mh$.
For the fourth, use the two exact computations
\[
 \frac{Q_\xi U_\xi}{Nh}=C_U Mh,
 \qquad
 \frac{Q_\xi\sqrt{\Lambda_\xi}}{Nh}=\sqrt{Q_\xi/N}.
\]

For $k\in\mathcal K_{\rm nr}(\xi,h)$,
$d_\xi(|k|)>2\vartheta_\xi$. The distance from $\Z$ is a
$1$-Lipschitz function, and $h\xi/k=A_\xi/|k|$, so
\[
 \operatorname{dist}(\Theta_k'(j),\Z)
 \ge d_\xi(|k|)-U_\xi\ge d_\xi(|k|)/2
                              \quad(0\le j\le2M).
\]
The derivative $\Theta_k'$ is continuous and monotone by
\eqref{eq3-add-mode-curvature}. Its image therefore lies in a
single component of $\R\setminus\Z$. With $\ell_k=\lfloor\Theta_k'(0)\rfloor\in\Z$, we have
\[
 \Theta_k'([0,2M])\subset
 [\ell_k+d_\xi(|k|)/2,\,\ell_k+1-d_\xi(|k|)/2].
\]
Applying
\Cref{lem3-add-derivatives}(i),
\eqref{eq3-add-amplitude-variation}, and Abel summation gives
\[
 \left|\sum_{j\in\mathcal J_k}b_k(j)e(\Theta_k(j))\right|
       \le C_{\mathfrak d}d_\xi(|k|)^{-1}.
\]
Use the reciprocal-distance estimate of
\Cref{lem3-add-resonance} and
$(N|k|)^{-1/2}\asymp_{\mathfrak d}(NQ_\xi)^{-1/2}$ to conclude
\begin{equation}
 \left|\sum_{k\in\mathcal K_{\rm nr}(\xi,h)}\mathcal S_k(\xi,h)\right|
 \lesssim_{\mathfrak d,\epsilon'}
 \frac{N^{\epsilon'}}{\sqrt{NQ_\xi}}
       \left(\vartheta_\xi^{-1}
                     +Q_\xi\log(2/\vartheta_\xi)\right).
 \label{eq3-add-nonresonant-before}
\end{equation}
Since $\vartheta_\xi\ge8\sqrt{\Lambda_\xi}$,
\[
 \frac1{\sqrt{NQ_\xi}\,\vartheta_\xi}\le\frac1{8Nh}.
\]
For the logarithm term, \eqref{eq3-add-stationary-k-range} and
$\tau\ge\tau_*$ give
\[
 \sqrt{\Lambda_\xi}=\frac{\tau}{\sqrt{NQ_\xi}}
 \ge\frac{\tau_*}{\sqrt{2r_+}}N^{-(1+\theta)/2}
 \ge\frac{\tau_*}{\sqrt{2r_+}}N^{-1}.
\]
Thus $\log(2/\vartheta_\xi)\le C_{\mathfrak d}\log(2N)$.
Given the exponent $\epsilon>0$ in
\eqref{eq3-add-offdiagonal-goal}, use $\epsilon'=\epsilon/2$ in
both counting estimates and
$\log(2N)\le C_\epsilon N^{\epsilon/2}$.
\eqref{eq3-add-resonant-total} and
\eqref{eq3-add-nonresonant-before} then bound $$\sum_{k\in\mathcal K(\xi)}\mathcal S_k(\xi,h)$$
by
\[
 C_{\mathfrak d,\epsilon}N^\epsilon
       \left(Mh+\frac1{Nh}+\sqrt{M/N}\right).
\]
Adding both remainders in \eqref{eq3-add-remainder-sum} proves
\eqref{eq3-add-offdiagonal-goal}. 

\medskip
\noindent\textbf{Step 6. The range $0<\tau=Nh<\tau_*$.}
For each $k\in\mathcal K(\xi)$,
\[
 \Theta_k'(j)=-\tau\psi_h(u_k(j))
       \in[-b_\psi\tau,-a_\psi\tau]
       \subset[-1+a_\psi\tau,-a_\psi\tau].
\]
For the last inclusion, use
$(a_\psi+b_\psi)\tau\le2b_\psi\tau_*\le1/2$.
Applying
\Cref{lem3-add-derivatives}(i),
\eqref{eq3-add-amplitude-variation}, and Abel summation gives
$$ \left|\sum_{j\in\mathcal J_k}b_k(j)e(\Theta_k(j))\right|
\le C_{\mathfrak d}\tau^{-1}.$$
On the other hand, \eqref{eq3-add-amplitude-variation} gives
$$ \left|\sum_{j\in\mathcal J_k}b_k(j)e(\Theta_k(j))\right|
\le |\mathcal J_k|\|b_k\|_{L^\infty([0,2M])}\le C_{\mathfrak d} M$$

Furthermore, \eqref{eq3-add-stationary-count} implies
\[
 \sum_{k\in\mathcal K(\xi)}(N|k|)^{-1/2}
 \le C_{\mathfrak d}\frac{Q_\xi}{\sqrt{NQ_\xi}}
 \le C_{\mathfrak d}\sqrt{M/N}.
\]
Multiplying these bounds and adding
\eqref{eq3-add-remainder-sum} proves
\eqref{eq3-add-tiny-offdiagonal}. 
This completes the proof of \Cref{prop3:offdiagonal}.
\end{proof}

\subsection{The diagonal contribution}\label{sec3:diagonal}

The proof of \Cref{prop3:offdiagonal} used $k\ne0$ to
retain a second-derivative lower bound independent of $h$.
For $k=0$, we instead sum over the integer index first and then
use the inverse of $\psi_h$. 
Let $\Omega=|\omega|^2$, extended by zero to $\R$. By
\eqref{eq3-add-mode-expansion},
\begin{equation}
 \mathcal D_h(\xi)=\sum_{m\in\Z}\Omega(m/M)
 \int_{I_0}w_h(u)e\bigl(\xi u-\tau m\psi_h(u)\bigr)\,du,
 \qquad \tau=Nh>0.
 \label{eq3-add-diagonal-original}
\end{equation}
Here and throughout this subsection, $\tau=Nh$. And, we shall use some notations from the previous subsection.

First suppose $0<\tau<\tau_*$. For each fixed $u\in I_0$, the
phase in the integer variable $m$ has slope $-\tau\psi_h(u)$,
which lies in $(-1/2,0)$ at distance at least $a_\psi\tau$ from
$\Z$. The finite geometric progression sum over consecutive relevant integers is
bounded both by $3M$ and by
$2/|1-e(-\tau\psi_h(u))|\le C_{\mathfrak d}/\tau$.
The smooth weight satisfies
\[
 \mathcal V_\R\bigl(\Omega(\cdot/M)\bigr)
 \le\|\Omega\|_\infty+\int_\R|\Omega'(v)|\,dv,
\]
independently of $M$. Abel summation, followed by
$\|w_h\|_{L^1(\R)}\le\|\eta\|_{L^2(\R)}^2$, therefore gives
\begin{equation}
 \sup_{\xi\in\R}|\mathcal D_h(\xi)|
       \lesssim_{\mathfrak d}\min\{M,\tau^{-1}\}.
 \label{eq3-add-tiny-diagonal}
\end{equation}
The integral of $\min\{M,\tau^{-1}\}$ on $(0,\tau_*)$ is
$M\tau_*\le1$ if $M\tau_*\le1$. Otherwise it equals
\[
 \int_0^{1/M}M\,d\tau+\int_{1/M}^{\tau_*}\frac{d\tau}{\tau}
 =1+\log(M\tau_*)\le1+\log M.
\]
Together with \eqref{eq3-add-tiny-offdiagonal} and
$\sqrt{M/N}\le1$, this proves
\begin{equation}
 \frac1L\int_0^{\tau_*}\sup_{\xi\in\R}
          |\mathcal A_{\tau/N}(\xi)|\,d\tau
 \lesssim_{\mathfrak d}\frac{\log(2M)}L+MN^{-3/2}.
 \label{eq3-add-tiny-integral}
\end{equation}
The use of \eqref{eq3-add-tiny-integral} requires $N\ge N_{\rm off}$.

\begin{prop}\label[prop]{prop3:diagonal}
There is $h_0=h_0(\mathfrak d)\in(0,1]$ such that, for every integer
$N\ge2$, real $M\in[1,N^\theta]$, and
$0<h\le\min\{H,h_0\}$,
\begin{equation}
 \sup_{\xi\in\R}|\mathcal D_h(\xi)|
       \lesssim_{\mathfrak d}1+\tau^{-1},\qquad \tau=Nh.
 \label{eq3-add-diagonal-absolute}
\end{equation}
For each fixed integer $d\ge3$, there is
$h_d=h_d(\mathfrak d,d)\in(0,h_0]$ such that, if also
$h\le h_d$ and $\tau\ge1$, then
\begin{equation}
 \sup_{\xi\in\R}|\mathcal D_h(\xi)|
 \lesssim_{\mathfrak d,d}\tau^{-1}
       +\left(\frac M{\tau^{d-1}}\right)^{1/(2^d-2)}
       +(M\tau)^{-1}.
 \label{eq3-add-diagonal-high-derivative}
\end{equation}
\end{prop}
\begin{proof}
The proof is divided into three steps.

\medskip
\noindent\textbf{Step 1. An inverse on a common open interval.}
Retain the intervals $I_0,I_1,I_2,J_\psi$ fixed in the proof of
\Cref{prop3:offdiagonal}, and put
\[
 J=J_\psi,\qquad J^+=[r_-/2,2r_+],\qquad
 \mathcal O_{\rm inv}=(r_-/4,4r_+).
\]
The last interval is open and contains the closed interval $J^+$.
For $u\in I_1$ and $0\le h\le1$,
\[
 \psi_h'(u)=\int_0^1\psi''(u+sh)\,ds<0,\qquad
 |\psi_h'(u)|\ge\min_{I_2}|\psi''|=:m_2>0.
\]
For each integer $s\ge0,0\le h\le 1$, differentiation under the integral in
\eqref{eq3-add-divided-difference} gives
\begin{equation}\label{eq3-3-0}
	\sup_{u\in I_1}|\psi_h^{(s)}(u)-\psi^{(s+1)}(u)|
	\le\frac h2\max_{I_2}|\psi^{(s+2)}|
	\le C_{\mathfrak d,s}h.
\end{equation}
Note that $\psi'(I_0)\subset\operatorname{int}J$.
Choose a closed interval $J_0\subset\operatorname{int}J$ containing
$\psi'(I_0)$ in its interior. The uniform estimate \eqref{eq3-3-0}, with $s=0$,
permits a single $h_0\in(0,1]$ such that
\[
 \mathcal O_{\rm inv}\subset\psi_h((a_1,b_1)),\qquad
 \psi_h(I_0)\subset\operatorname{int}J_0
                                \quad(0\le h\le h_0).
\]
Both $J_0$ and $h_0$ are chosen from $\rho,I_0,I_1,I_2$ and hence
from $\mathfrak d$. For $0\le h\le h_0$, the strict monotonicity of $\psi_h$ and the inverse function
theorem define a smooth real function
\[
 \Psi_h:\mathcal O_{\rm inv}\to(a_1,b_1),\qquad
 \psi_h(\Psi_h(v))=v.
\]

We verify uniformity of all inverse derivatives. For $v\in J^+$,
write $x_h(v)=\Psi_h(v)$ and $x_0(v)=\Psi_0(v)$. The mean value
theorem applied to $\psi_0=\psi'$ gives
\[
 m_0|x_h(v)-x_0(v)|
 \le|\psi_0(x_h(v))-\psi_0(x_0(v))|
 =|\psi_0(x_h(v))-\psi_h(x_h(v))|
 \le\|\psi_0-\psi_h\|_{C^0(I_1)}.
\]
Thus $\Psi_h\to\Psi_0$ uniformly on $J^+$ by \eqref{eq3-3-0}.
On the open interval $(0,\infty)$, for $0\le h\le h_0$, define
\[
 R_{1,h}=1/\psi_h',\qquad R_{s+1,h}=R_{s,h}'/\psi_h'\quad(s\ge1).
\]
Differentiating the inverse identity proves by induction that
$\Psi_h^{(s)}=R_{s,h}\circ\Psi_h$ on $\mathcal O_{\rm inv}$.


For each fixed integer $s\ge1$, the recurrence defining $R_{s,h}$
involves only $\psi_h',\ldots,\psi_h^{(s)}$ and divisions by
powers of $\psi_h'$. Since $|\psi_h'|\ge m_2>0$ on $I_1$,
the uniform derivative bounds and the convergence
$\psi_h\to\psi_0$ in $C^s(I_1)$ imply
\[
\sup_{0\le h\le h_0}\|R_{s,h}\|_{C^0(I_1)}
\le C_{\mathfrak d,s},
\qquad
\|R_{s,h}-R_{s,0}\|_{C^0(I_1)}\longrightarrow0
\quad(h\downarrow0).
\]
Recall that
$\Psi_h^{(s)}=R_{s,h}\circ\Psi_h$ and
$\Psi_h(J^+)\subset I_1$. Hence
\[
\|\Psi_h^{(s)}\|_{C^0(J^+)}
\le \|R_{s,h}\|_{C^0(I_1)}
\le C_{\mathfrak d,s}.
\]
To prove convergence, use the mean value theorem for the
smooth function $R_{s,0}$ on $I_1$:
\[
\begin{aligned}
\|\Psi_h^{(s)}-\Psi_0^{(s)}\|_{C^0(J^+)}
\le \|R_{s,h}-R_{s,0}\|_{C^0(I_1)}
+
\|R_{s,0}'\|_{C^0(I_1)}
\|\Psi_h-\Psi_0\|_{C^0(J^+)}.
\end{aligned}
\]
Both terms tend to zero as $h\downarrow0$. 
Applying these bounds at orders $1,\ldots,s$, together with the
order-zero bounds, gives
\[
\sup_{0\le h\le h_0}\|\Psi_h\|_{C^s(J^+)}
\le C_{\mathfrak d,s},
\qquad
\|\Psi_h-\Psi_0\|_{C^s(J^+)}\longrightarrow0
\quad(h\downarrow0)
\]
for every integer $s\ge0$. 

Define
\begin{equation}
 a_h(v)=
 \begin{cases}
 w_h(\Psi_h(v))/|\psi_h'(\Psi_h(v))|,&v\in\mathcal O_{\rm inv},\\
 0,&v\notin\mathcal O_{\rm inv}.
 \end{cases}
 \label{eq3-add-inverse-amplitude}
\end{equation}
A nonzero numerator requires $\Psi_h(v)\in I_0$ and hence
$v\in\psi_h(I_0)\subset\operatorname{int}J_0$.
Thus $a_h$ extends smoothly by zero and
$$\sup_{0\le h\le h_0}\|a_h\|_{C^s(\R)}\le C_{\mathfrak d,s}$$ for every fixed $s$.


Choose a real-valued function $\chi_{\rm inv}\in C_c^\infty(\R)$,
depending only on $J$ and $J^+$, such that
\[
\operatorname{supp}\chi_{\rm inv}\subset\operatorname{int}J^+,
\qquad
\chi_{\rm inv}=1 \quad\text{on an open neighborhood of }J.
\]
For $0\le h\le h_0$, define
\[
\Psi_h^{\rm ext}(v)=
\begin{cases}
\chi_{\rm inv}(v)\Psi_h(v),&v\in\mathcal O_{\rm inv},\\
0,&v\notin\mathcal O_{\rm inv}.
\end{cases}
\]
Since $\operatorname{supp}\chi_{\rm inv}$ is compactly contained
in $\mathcal O_{\rm inv}$, this defines a smooth function on
$\R$ with compact support. The product rule and the uniform
derivative bounds for $\Psi_h$ on $J^+$ give
\[
\sup_{0\le h\le h_0}
\|\Psi_h^{\rm ext}\|_{C^s(\R)}
\le C_{\mathfrak d,s}
\qquad(s\ge 0).
\]

In \eqref{eq3-add-diagonal-original}, the decreasing substitution
$v=\psi_h(u)$ yields
\begin{equation}
 \mathcal D_h(\xi)=\sum_{m\in\Z}\Omega(m/M)
 \int_\R a_h(v)e\bigl(\xi\Psi_h^{\rm ext}(v)-\tau mv\bigr)\,dv.
 \label{eq3-add-diagonal-inverse}
\end{equation}

\medskip
\noindent\textbf{Step 2. Poisson summation and the relevant frequencies.}
Apply \Cref{lem3-add-dual-formula} with
$a=a_h$ and $\Psi=\Psi_h^{\rm ext}$.
Their supports and derivative upper bounds are supplied by the estimates
above. Also on $J$, 
\[
 |(\Psi_h^{\rm ext})'(v)|=|\Psi_h'(v)|
 =\frac1{|\psi_h'(\Psi_h(v))|}
 \ge\left(\max_{I_2}|\psi''|\right)^{-1}>0.
\]
Thus the fixed data in \Cref{lem3-add-dual-formula} depend only on $\mathfrak d$, and
\eqref{eqB:absolute} proves
\eqref{eq3-add-diagonal-absolute}.
For $\tau\ge1$ and $|\xi|\le C_\infty M\tau$, where
$C_\infty=C_\infty(\mathfrak d)$ is chosen as in
\Cref{lem3-add-dual-formula}, \eqref{eqB:poisson-formula}
gives
\begin{equation}
 \begin{split}
 \mathcal D_h(\xi)=\frac1\tau
 \sum_{\ell\in\tau J\cap\Z}a_h(\ell/\tau)
 \Omega\left(\frac{\xi\Psi_h'(\ell/\tau)}{M\tau}\right)
 e\bigl(\xi\Psi_h(\ell/\tau)\bigr)
 +O_{\mathfrak d}((M\tau)^{-1}).
 \end{split}
 \label{eq3-add-dual-formula}
\end{equation}
There are fixed $0<a_{\rm inv}\le b_{\rm inv}$, depending only on
$\mathfrak d$, such that
$a_{\rm inv}\le|\Psi_h'|\le b_{\rm inv}$ on $J$.
Since $\Psi_h'<0$ and $\operatorname{supp}\Omega\subset(1/2,2)$,
a nonzero leading sum in \eqref{eq3-add-dual-formula} requires
\[
 \xi<0,\qquad
 \frac1{2b_{\rm inv}}M\tau<|\xi|<\frac2{a_{\rm inv}}M\tau.
\]
For $\tau\ge1$ and $|\xi|>C_\infty M\tau$, \eqref{eqB:large-frequency}
gives $|\mathcal D_h(\xi)|\le C_{\mathfrak d}(M\tau)^{-1}$.

\medskip
\noindent\textbf{Step 3. Derivative estimates for the dual sum.}
On $\mathcal O_{\rm inv}$, the limiting inverse is
\begin{equation}
 \Psi_0(v)=(v/\rho)^{1/(\rho-1)}.
 \label{eq3-add-limiting-inverse}
\end{equation}
Write $a_\rho=1/(\rho-1)<0$. For every integer $s\ge1$,
\[
 \Psi_0^{(s)}(v)
 =\rho^{-a_\rho}a_\rho(a_\rho-1)\cdots(a_\rho-s+1)v^{a_\rho-s}.
\]
Its sign is $(-1)^s$ and
$b_s:=\min_J|\Psi_0^{(s)}|>0$. Also set
$B_s:=\max_J|\Psi_0^{(s)}|$.
Fix $d\ge3$. Convergence in $C^d(J^+)$ permits one
$h_d=h_d(\mathfrak d,d)\in(0,h_0]$ such that, for $0\le h\le h_d$,
$\sup_J|\Psi_h^{(s)}-\Psi_0^{(s)}|\le b_s/2$ simultaneously for
$1\le s\le d$. Hence,
\begin{equation}
 c_s\le|\Psi_h^{(s)}(v)|\le C_s
       \quad(v\in J,\ 0\le h\le h_d,\ 1\le s\le d),
 \quad c_s=b_s/2,\quad C_s=B_s+b_s/2.
 \label{eq3-add-inverse-derivatives}
\end{equation}
These constants depend only on $\rho,J,s$. Each derivative retains sign $(-1)^s$.

Suppose $\tau\ge2$ and the leading sum in
\eqref{eq3-add-dual-formula} is not identically zero.
Set $\zeta_{M,\tau,\xi}=\xi/(M\tau)$. The frequency restrictions
above place this real parameter in a fixed compact subset of
$(-\infty,0)$, separated from zero.
For such a parameter $\zeta$, define the smooth function
\[
 b_{h,\zeta}(v)=a_h(v)
       \Omega\bigl(\zeta(\Psi_h^{\rm ext})'(v)\bigr),\qquad v\in\R,0\le h\le h_d.
\]
Its support is contained in $J_0$, and the product rule gives
\[
 b_{h,\zeta}'=a_h'\Omega\bigl(\zeta(\Psi_h^{\rm ext})'\bigr)
 +a_h\Omega'\bigl(\zeta(\Psi_h^{\rm ext})'\bigr)
                                  \zeta(\Psi_h^{\rm ext})''.
\]
Every factor is uniformly bounded using $\mathfrak d$. Therefore
\[
 \mathcal V_\R(b_{h,\zeta})
 \le\|b_{h,\zeta}\|_{L^{\infty}(\R)}+|J_0|\sup_\R|b_{h,\zeta}'|
 \le C_{\mathfrak d}.
\]
For $\tau>0$, the substitution $x=\tau v$ gives that 
\[
\begin{aligned}
\|b_{h,\zeta}(\cdot/\tau)\|_{L^\infty(\tau J)}
&=\|b_{h,\zeta}\|_{L^\infty(J)},\\
\int_{\tau J}
\left|\frac{d}{dx}b_{h,\zeta}(x/\tau)\right|\,dx
&=\frac1\tau\int_{\tau J}|b_{h,\zeta}'(x/\tau)|\,dx
=\int_J|b_{h,\zeta}'(v)|\,dv.
\end{aligned}
\]
Consequently,
\begin{equation}\label{eq3-3-1}
	\mathcal V_{\tau J}\bigl(b_{h,\zeta}(\cdot/\tau)\bigr)
	=\mathcal V_J(b_{h,\zeta})
	\le C_{\mathfrak d}.
\end{equation}

Based on the form of the leading sum in \eqref{eq3-add-dual-formula}, we define the real phase by
\[
 q_{h,\tau,\xi}:\tau\mathcal O_{\rm inv}\to\R,
 \qquad q_{h,\tau,\xi}(x)=\xi\Psi_h(x/\tau),\qquad 0\le h\le h_d.
\]
\eqref{eq3-add-inverse-derivatives} implies that 
\[
 q_{h,\tau,\xi}^{(s)}(x)=\xi\tau^{-s}\Psi_h^{(s)}(x/\tau),\qquad
 |q_{h,\tau,\xi}^{(s)}(x)|\asymp_{\mathfrak d,s}|\xi|\tau^{-s}
                    \quad(1\le s\le d,\ x\in\tau J).
\]

We apply \Cref{lem3-add-derivatives}(iii) with
$L=\tau$, $\mathcal F=|\xi|$,
$a_*=r_-$, and $b_*=r_+$. 
Applying
\Cref{lem3-add-derivatives}(iii),
\eqref{eq3-3-1}, and Abel summation gives
\[
 \begin{split}
 &\frac1\tau\left|\sum_{\ell\in\tau J\cap\Z}
 b_{h,\zeta_{M,\tau,\xi}}(\ell/\tau)
                         e(q_{h,\tau,\xi}(\ell))\right|
 \lesssim_{\mathfrak d,d}
   \tau^{-1}+(|\xi|\tau^{-d})^{1/(2^d-2)}+|\xi|^{-1}\\
 &\hspace{8cm}\lesssim_{\mathfrak d,d}
   \tau^{-1}+(M/\tau^{d-1})^{1/(2^d-2)}+(M\tau)^{-1}.
 \end{split}
\]

The second inequality uses $|\xi|\asymp_{\mathfrak d}M\tau$.
The amplitude is zero outside $\tau J_0$, so it is the leading
sum in \eqref{eq3-add-dual-formula}. Adding its remainder proves
\eqref{eq3-add-diagonal-high-derivative} for the present frequencies.
The complementary frequencies were already bounded by
$C_{\mathfrak d}(M\tau)^{-1}$.
For $1\le\tau<2$, \eqref{eq3-add-diagonal-absolute} gives
at most $2C_{\mathfrak d}$.
Then by $\tau^{-1}>1/2$, increasing $C_{\mathfrak d,d}$ covers this interval and proves
\Cref{prop3:diagonal}.
\end{proof}

\subsection{Completion of the proof of \texorpdfstring{\Cref{thm3-1}}{Theorem~\getrefnumber{thm3-1}}}\label{sec3-13}

Take $d=d_\theta$ and $P=2^d\ge 16$, as in
\eqref{eq3-add-parameters}. These integers are fixed by $\theta$.
Choose an integer $N_0$ at least
\[
 \max\left\{N_{\rm off},\,
 \left\lceil h_d^{-1/(1-\sigma)}\right\rceil,\,
 \left\lceil 2^{2/\sigma}\right\rceil\right\}.
\]
Then, for every $N\ge N_0>1$, $H\le h_d$ and
$$L>N^{\sigma/2}\ge 2\ge 1\ge \tau_{*}>0.$$

We estimate the normalized scalar integral in
\eqref{eq3-add-dual-time-reduction}. First,
\eqref{eq3-add-tiny-integral} gives
\begin{equation}
 \frac1L\int_0^{\tau_*}\sup_{\xi\in\R}
       |\mathcal A_{\tau/N}(\xi)|\,d\tau
 \lesssim_{\mathfrak d}N^{-\sigma}\log(2N)+MN^{-3/2}
 \lesssim_{\mathfrak d}N^{-\sigma/2}.
 \label{eq3-add-final-tiny}
\end{equation}

For the off-diagonal contribution on $[\tau_*,L]$, substitute
$h=\tau/N$ into \eqref{eq3-add-offdiagonal-goal}. The three terms
have the respective normalized integrals:
\[
 \begin{split}
 \frac1L\int_{\tau_*}^L Mh\,d\tau
   \le\frac{ML}{2N},\qquad
 \frac1L\int_{\tau_*}^L\frac{d\tau}{Nh}
   =\frac{\log(L/\tau_*)}{L},\qquad
 \frac1L\int_{\tau_*}^L\sqrt{M/N}\,d\tau\le\sqrt{M/N}.
 \end{split}
\]
Consequently, for every fixed $\epsilon>0$,
\begin{equation}
 \begin{split}
 \frac1L\int_{\tau_*}^L\sup_{\xi\in\R}
                  |\mathcal O_{\tau/N}(\xi)|\,d\tau
 \lesssim_{\mathfrak d,\epsilon}N^\epsilon
       \left(\frac{ML}{N}+\frac{\log(2L)}L+\sqrt{\frac MN}\right)
 \lesssim_{\mathfrak d,\epsilon}
                          N^{-\sigma+\epsilon}\log(2N).
 \end{split}
 \label{eq3-add-final-offdiagonal}
\end{equation}
Choose $\epsilon=\sigma/4$.
Then the quantity in \eqref{eq3-add-final-offdiagonal} is at most
$C_{\mathfrak d}N^{-\sigma/2}$.

On $[\tau_*,N^{\sigma/2}]$,
\eqref{eq3-add-diagonal-absolute} is at most
$C_{\mathfrak d}(1+\tau_*^{-1})$. Hence,
\begin{equation}
 \frac1L\int_{\tau_*}^{N^{\sigma/2}}\sup_{\xi\in\R}
               |\mathcal D_{\tau/N}(\xi)|\,d\tau
 \lesssim_{\mathfrak d}N^{-\sigma/2}.
 \label{eq3-add-final-moderate}
\end{equation}
For $1<N^{\sigma/2}\le\tau\le L$, the definition of $d$ gives
$(d-1)\sigma/2\ge1$. Thus
\[
 \frac M{\tau^{d-1}}
 \le N^{\theta-(d-1)\sigma/2}\le N^{\theta-1}=N^{-2\sigma}.
\]
At present, all assumptions of \eqref{eq3-add-diagonal-high-derivative} hold,
including $h=\tau/N\le H\le h_d$ and $\tau\ge1$.
Its three terms satisfy
\begin{align*}
 \tau^{-1}&\le N^{-\sigma/2}\le N^{-\sigma/P},\\
 (M/\tau^{d-1})^{1/(P-2)}
   &\le N^{-2\sigma/(P-2)}\le N^{-\sigma/P},\\
 (M\tau)^{-1}&\le\tau^{-1}\le N^{-\sigma/P}.
\end{align*}
By the above, we have 
\begin{equation}
 \frac1L\int_{N^{\sigma/2}}^L\sup_{\xi\in\R}
                 |\mathcal D_{\tau/N}(\xi)|\,d\tau
 \lesssim_{\mathfrak d}N^{-\sigma/P}.
 \label{eq3-add-final-large}
\end{equation}

Use $$\sup_\xi|\mathcal A_h(\xi)|
\le\sup_\xi|\mathcal O_h(\xi)|+\sup_\xi|\mathcal D_h(\xi)|$$
only on $[\tau_*,L]$, and use
\eqref{eq3-add-final-tiny} on $(0,\tau_*)$.
\eqref{eq3-add-final-offdiagonal},
\eqref{eq3-add-final-moderate}, and \eqref{eq3-add-final-large}
then give \eqref{eq3-add-main-bound} for $N\ge N_0$ by
\eqref{eq3-add-dual-time-reduction}.

For $2\le N<N_0$, there are at most $2N_0^\theta+1$ integers
with $\omega(m/M)\ne0$. The triangle inequality gives
\[
 \left\|\int_\R V(u)\,du\right\|_2
 \lesssim \|\eta\|_{L^1(\R)}\|\omega\|_{L^\infty(\R)}N_0^\theta\lesssim_{\mathfrak d} N_{0}^{\theta+\sigma/P}N_{0}^{-\sigma/P}\lesssim_{\mathfrak d} N^{-\sigma/P}.
\]
This completes the proof of \Cref{thm3-1}.\hfill$\square$

\appendix
\section{Analytic estimates}\label[appendix]{app:derivatives}

We state some estimates for exponential sums used in \Cref{sec3:main}. By Abel summation, these estimates also work for weights with bounded variation norms.

\begin{lemma}\label[lemma]{lem3-add-derivatives}
Let $I=[a_-,a_+]$, where $a_-<a_+$ are real, and let
$\mathcal O\subset\R$ be an open interval containing $I$.
The real phase $q:\mathcal O\to\R$ has the regularity specified
in the following assertions. 
\begin{itemize}
\item[(i)] Suppose $q\in C^1(\mathcal O;\R)$ and $q'$ is
monotone on $I$. If $q'(I)\subset[r+\delta,r+1-\delta]$ for
some $r\in\Z$ and $0<\delta\le1/2$, then
\[
 \left|\sum_{n\in I\cap\Z}e(q(n))\right|
                         \lesssim \delta^{-1}.
\]
\item[(ii)] Suppose $q\in C^2(\mathcal O;\R)$,
$\lambda\le|q''(x)|\le C_0\lambda$ for every $x\in I$, where
$\lambda>0$, $C_0\ge1$, and $|I|\le L$ with $L\ge1$. Then
\[
 \left|\sum_{n\in I\cap\Z}e(q(n))\right|
          \lesssim_{C_0}\bigl(L\lambda^{1/2}+\lambda^{-1/2}\bigr).
\]
\item[(iii)] Fix $0<a_*<b_*$, an integer $d\ge3$, and constants
$0<c_s\le C_s$ for $1\le s\le d$; denote these data by
$\mathfrak a_d=(a_*,b_*,d,(c_s)_{s=1}^d,(C_s)_{s=1}^d)$.
Suppose $L\ge2$, $\mathcal F>0$, $I\subset[a_*L,b_*L]$,
and $q\in C^d(\mathcal O;\R)$ satisfies
\[
 c_s\mathcal F L^{-s}\le|q^{(s)}(x)|\le C_s\mathcal F L^{-s}
                         \quad(x\in I,\ 1\le s\le d).
\]
Then, with $P_d=2^d$,
\begin{equation}
 \frac1L\left|\sum_{n\in I\cap\Z}e(q(n))\right|
 \lesssim_{\mathfrak a_d}
 L^{-1}+(\mathcal F L^{-d})^{1/(P_d-2)}+\mathcal F^{-1}.
 \label{eq3-add-derivative-test}
\end{equation}
\end{itemize}
The implicit constants in (i), (ii), and (iii) are independent of
$\mathcal O,I,q,r,\delta,\lambda,L,\mathcal F$.
\footnote
{Every estimate holds with the original scale when the sum is
restricted to a closed subinterval of $I$, including a singleton.}
\end{lemma}
\begin{proof}
	
\medskip	
\noindent\textbf{Part I. First derivative. }
Empty and one-term sums satisfy the claimed bound. Otherwise, let
$n_-=\min(I\cap\Z)<n_+=\max(I\cap \Z)$. Replacing $q(t)$ by $q(t)-rt$ does not
change $e(q(n))$ at integer $n$. Thus we may assume
$\delta\le q'\le1-\delta$ on $I$.
For $n_-\le n<n_+$, define
\[
 t_n=q(n+1)-q(n)=\int_0^1q'(n+s)\,ds,\qquad z_n=e(q(n)).
\]
The sequence $t_n$ is monotone and lies in $[\delta,1-\delta]$.
For $0<t<1$, let
\[
 B(t)=(1-e(t))^{-1}=\frac12+\frac i2\cot(\pi t).
\]
The identity $z_n=B(t_n)(z_n-z_{n+1})$, summed over
$n_-\le n<n_+$, and the remaining term $z_{n_+}$ give
\[
 \begin{split}
 \sum_{n=n_-}^{n_+}z_n={}&B(t_{n_-})z_{n_-}
       +(1-B(t_{n_+-1}))z_{n_+}
 +\sum_{n=n_-+1}^{n_+-1}
       (B(t_n)-B(t_{n-1}))z_n.
 \end{split}
\]
Since $|e(q(n))|=1$, it is enough to bound the coefficients here.
On $[\delta,1-\delta]$,
\[
 \sup|B|\le\frac1{2\sin(\pi\delta)}\lesssim \delta^{-1},
 \qquad
 \int_\delta^{1-\delta}|B'(t)|\,dt
       =\cot(\pi\delta)\lesssim \delta^{-1}.
\]
Monotonicity of $t_n$ gives
$$\left| \sum_{n=n_-}^{n_+}z_n\right|\lesssim \delta^{-1}+\int_\delta^{1-\delta}|B'(t)|\,dt\lesssim \delta^{-1}.$$
This proves \Cref{lem3-add-derivatives}(i).

\medskip
\noindent\textbf{Part II. Second derivative.}
Empty and one-term sums satisfy the claimed bound. Otherwise, let
$n_-:=\min(I\cap\Z)<n_+:=\max(I\cap \Z)$. Put $1\le D:=n_+-n_-\le L$.
The translated phase $x\mapsto q(n_-+x)$ is defined on a suitable open
neighborhood of $[0,D]$ and has the same derivative bounds on $[0,D]$.
The sign of $q''$ is constant, since it is continuous and nonzero.
By the case $k=2$ of \cite[Equation~(1)]{HeathBrown2017}, we have 
\[
\left|\sum_{n=1}^{D}e(q(n+n_{-}))\right|
\lesssim_{C_0}\bigl(D\sqrt{\lambda}+\lambda^{-1/2}\bigr).
\]
Adding the index $n=0$ costs at most one. This can be absorbed by
$L\sqrt\lambda+\lambda^{-1/2}$. 
This proves \Cref{lem3-add-derivatives}(ii).

\medskip
\noindent\textbf{Part III. Higher derivatives.}
Set
$M_0=a_*L$. If $M_0\le2$, then
\[
\begin{aligned}
\frac1L\left|\sum_{n\in I\cap\Z}e(q(n))\right|
\le \frac{(b_*-a_*)L+1}{L}
\le \left(1+\frac{2(b_*-a_*)}{a_*}\right)L^{-1}.
\end{aligned}
\]
A singleton also contributes at most $L^{-1}$, and an empty sum
is zero. We may therefore assume that $M_0>2$ and $|[a_{-},a_{+}]\cap \Z|>1$.

Choose $C_{\rm dil}=1+b_*/a_*$. Then
$
[a_-,a_+)\subset[M_0,C_{\rm dil}M_0).
$
Since $L^{-s}=a_*^sM_0^{-s}$, the derivative assumptions imply
\[
c_sa_*^s\mathcal F M_0^{-s}
\le |q^{(s)}(x)|
\le C_sa_*^s\mathcal F M_0^{-s}
\qquad(x\in I,\ 1\le s\le d).
\]
Thus \cite[Theorem~1.2]{LST2025} applies to the sum over
$a_-\le n<a_+$ with $l=d-2$. It gives
\[
\left|\sum_{a_-\le n<a_+}e(q(n))\right|
\lesssim_{\mathfrak a_d}
\left(
\mathcal F^{1/(P_d-2)}M_0^{\,1-d/(P_d-2)}
+\mathcal F^{-1}M_0
\right).
\]
Substituting $M_0=a_*L$ and dividing by $L$, we obtain
\[
\frac1L\left|\sum_{a_-\le n<a_+}e(q(n))\right|
\lesssim_{\mathfrak a_d}
\left(
(\mathcal F L^{-d})^{1/(P_d-2)}
+\mathcal F^{-1}
\right).
\]

Passing from $[a_-,a_+)$ to $[a_-,a_+]$ adds at most one summand
of absolute value one. Hence,
\[
\frac1L\left|\sum_{n\in I\cap\Z}e(q(n))\right|
\lesssim_{\mathfrak a_d}
\left(
L^{-1}
+(\mathcal F L^{-d})^{1/(P_d-2)}
+\mathcal F^{-1}
\right).
\]
This proves
\Cref{lem3-add-derivatives}(iii).
\end{proof}
\begin{lemma}\label[lemma]{lem3-add-stationary}
Let $I=[u_-,u_+]$, with $u_-<u_+$, and let
$I_a,I_*\subset(u_-,u_+)$ be fixed closed intervals.
Fix positive $A_0,B_0,c_0,C_0$, and denote these data by
$
 \mathfrak s=(I,I_a,I_*,A_0,B_0,c_0,C_0).
$
Suppose $a\in C_c^\infty(\R;\C)$ has support in $I_a$ and
$\|a\|_{C^6(\R)}\le A_0$. Let
$p\in C^\infty(\mathcal O;\R)$ on an open interval
$\mathcal O\supset I$, with
\[
 \|p\|_{C^8(I)}\le B_0,\qquad
 c_0\le|p''(u)|\le C_0\quad(u\in I).
\]
If $p'(u_p)=0$ for a point $u_p\in I_*$, then it
is unique on $I$. For any real $A\ge1$,
\begin{equation}
 \int_I a(u)e(Ap(u))\,du
 =A^{-1/2}e\bigl(Ap(u_p)+\operatorname{sgn}(p'')/8\bigr)
       \frac{a(u_p)}{|p''(u_p)|^{1/2}}+O_{\mathfrak s}(A^{-3/2}).
 \label{eq3-add-stationary-expansion}
\end{equation}
\end{lemma}
\begin{proof}
The proof is divided into three steps.

	\medskip
	\noindent\textbf{Step 1. Reduction to a quadratic phase.}
	Since $p''$ is continuous and never zero on $I$, its sign is
	constant. Put
	$
	\epsilon_p=\operatorname{sgn}(p''(u_p))\in\{-1,1\}.
$
	$p'$ is strictly monotone, so $u_p$ is its unique
	zero on $I$. For $u\in I$, define
	\[
	B_p(u)=2\epsilon_p\int_0^1(1-t)
	p''(u_p+t(u-u_p))\,dt,
	\qquad
	z_p(u)=(u-u_p)\sqrt{B_p(u)}.
	\]
	Taylor's formula and $p'(u_p)=0$ give
	\[
	c_0\le B_p(u)\le C_0,
	\qquad
	p(u)=p(u_p)+\frac{\epsilon_p}{2}z_p(u)^2.
	\]
	
	Differentiation gives
	\[
	z_p'(u)=
	\frac{\epsilon_p\displaystyle\int_0^1
		p''(u_p+t(u-u_p))\,dt}
	{\sqrt{B_p(u)}}.
	\]
	This identity also holds at $u=u_p$, where
	$z_p'(u_p)=\sqrt{|p''(u_p)|}$. Hence,
	\[
	\frac{c_0}{\sqrt{C_0}}
	\le z_p'(u)\le
	\frac{C_0}{\sqrt{c_0}}
	\qquad(u\in I).
	\]
	In particular, $z_p$ is strictly increasing. Write
	\[
	J_p=(z_p(u_-),z_p(u_+)),
	\qquad
	v_p:J_p\longrightarrow(u_-,u_+)
	\]
	for its smooth inverse.
	
	\medskip
	\noindent\textbf{Step 2. Bounds for the transformed amplitude.}
	For $0\le r\le6$,
	\[
	B_p^{(r)}(u)
	=2\epsilon_p\int_0^1(1-t)t^r
	p^{(r+2)}(u_p+t(u-u_p))\,dt.
	\]
	The bound $\|p\|_{C^8(I)}\le B_0$ therefore bounds all these
	derivatives. Since $B_p\ge c_0$, the definition of $z_p$ gives
	\[
	\|z_p\|_{C^6(I)}\lesssim_{\mathfrak s} 1.
	\]
	The inverse identity yields
	\[
	v_p'(z)=\frac1{z_p'(v_p(z))},
	\qquad
	v_p''(z)=-
	\frac{z_p''(v_p(z))}{(z_p'(v_p(z)))^3}.
	\]
	Repeated differentiation of the first identity, using the lower
	bound for $z_p'$ and the bounds for its higher derivatives, gives
	\[
	\max_{1\le r\le5}\sup_{z\in J_p}|v_p^{(r)}(z)|
	\lesssim_{\mathfrak s} 1.
	\]
	
	Define
	\[
	b_{a,p}(z)=
	\begin{cases}
	a(v_p(z))v_p'(z),&z\in J_p,\\
	0,&z\notin J_p.
	\end{cases}
	\]
	Its support is contained in $z_p(I_a)$, a compact subset of
	$J_p$. Thus $b_{a,p}$ is smooth. The product rule,
	the chain rule, and the preceding bounds imply
	\[
	\|b_{a,p}\|_{C^4(\R)}\lesssim_{\mathfrak s} 1.
	\]
	Moreover, $|z_p(u)|\le\sqrt{C_0}|I|$ for $u\in I$, so
	\[
	\operatorname{supp}b_{a,p}
	\subset[-\sqrt{C_0}|I|,\sqrt{C_0}|I|],
	\qquad
	\sum_{r=0}^{3}\|b_{a,p}^{(r)}\|_{L^2(\R)}
	\lesssim_{\mathfrak s} 1.
	\]
	Since $z_p(u_p)=0$, we also have
	$
	b_{a,p}(0)={a(u_p)}/{\sqrt{|p''(u_p)|}}.
	$
	
	\medskip
	\noindent\textbf{Step 3. Evaluation of the quadratic integral.}
	The substitution $z=z_p(u)$ gives
	\[
	\int_I a(u)e(Ap(u))\,du
	=
	e(Ap(u_p))
	\int_\R b_{a,p}(z)e(\epsilon_p A z^2/2)\,dz.
	\]
	Apply \cite[Theorem~2.2]{Blair2022} in
	dimension one, with $k=1$, scalar matrix
	$\epsilon_p$, and $\delta=(2\pi A)^{-1}$. It gives
	\[
	\begin{aligned}
	\int_\R b_{a,p}(z)e(\epsilon_p A z^2/2)\,dz
	&=A^{-1/2}e(\epsilon_p/8)b_{a,p}(0)
	+\mathcal R_A(a,p),\\
	|\mathcal R_A(a,p)|
	&\lesssim A^{-3/2}
	\sum_{r=0}^{3}\|b_{a,p}^{(r)}\|_{L^2(\R)}
	\lesssim_{\mathfrak s}A^{-3/2}.
	\end{aligned}
	\]
	Substituting the value of $b_{a,p}(0)$ proves this lemma.
\end{proof}
To conclude this section, we give a remark on nonstationary integration by parts.

\medskip
\noindent\textbf{Remark.}
Fix a closed interval $I=[u_-,u_+]$ with $u_-<u_+$, a closed
interval $K\subset(u_-,u_+)$, an integer $B\ge1$, and positive
constants $c_0,A_B,P_2,\ldots,P_{B+1}$. Write
$$
 \mathfrak n_B=(I,K,B,c_0,A_B,P_2,\ldots,P_{B+1}).
$$
Suppose $a\in C_c^\infty(\R;\C)$ has support in $K$ and
$\|a\|_{C^B(\R)}\le A_B$. Let $p$ be real and smooth on an
open interval $\mathcal O\supset I$, with
$|p'|\ge c_0$ and $|p^{(s)}|\le P_s$ on $I$ for
$2\le s\le B+1$. At this point, $I$ lies in one open component $\mathcal O_p$
of $\{u\in\mathcal O:p'(u)\ne0\}$.
On that component, define
\[
a_{p,0}=a,\qquad a_{p,j+1}
=-\frac d{du}(a_{p,j}/p')\quad(0\le j<B).
\]
Then we have
\begin{equation}
 \int_Ia(u)e(Ap(u))\,du
 =(2\pi iA)^{-B}\int_Ia_{p,B}(u)e(Ap(u))\,du
 =O_{\mathfrak n_B}(A^{-B})\qquad(A\ge1).
 \label{eq3-add-nonstationary}
\end{equation}

\noindent \textit{Proof of \eqref{eq3-add-nonstationary}.} 
Since
\[
\frac{d}{du}e(Ap(u))=2\pi iA\,p'(u)e(Ap(u)),
\]
one integration by parts gives
\[
\begin{aligned}
\int_I a(u)e(Ap(u))\,du
&=\frac1{2\pi iA}
\int_{u_-}^{u_+}\frac{a(u)}{p'(u)}
\frac{d}{du}e(Ap(u))\,du\\
&=\frac1{2\pi iA}
\left[\frac{a(u)}{p'(u)}e(Ap(u))\right]_{u_-}^{u_+}
-\frac1{2\pi iA}
\int_I\left(\frac{a(u)}{p'(u)}\right)'e(Ap(u))\,du\\
&=\frac1{2\pi iA}\int_I a_{p,1}(u)e(Ap(u))\,du.
\end{aligned}
\]
Repeating the same integration by parts
$B-1$ more times gives
\[
\int_I a(u)e(Ap(u))\,du
=\frac1{(2\pi iA)^B}
\int_I a_{p,B}(u)e(Ap(u))\,du,
\]
which is the exact equality in \eqref{eq3-add-nonstationary}.

For the estimate, set $Q_p=1/p'$ on $\mathcal O_p$.
Differentiation of $p'Q_p=1$ gives
\[
 Q_p^{(r)}=-\frac1{p'}\sum_{s=1}^r
              \binom rs p^{(s+1)}Q_p^{(r-s)}\qquad(1\le r\le B).
\]
Starting with $|Q_p(u)|\le c_0^{-1},u\in I$, the above recurrence bounds
$\|Q_p\|_{C^B(I)}$ using only $\mathfrak n_B$.
The recurrence for $a_{p,j}$ then gives
$\|a_{p,B}\|_{L^{1}(\R)}\lesssim_{\mathfrak n_B} 1$.
This proves \eqref{eq3-add-nonstationary}.\hfill$\square$

\section{A smooth Poisson transformation}\label[appendix]{app:oscillatory}

\Cref{lem3-add-dual-formula} isolates the parameter-uniform
Poisson formula used in \Cref{prop3:diagonal}.

\begin{lemma}\label[lemma]{lem3-add-dual-formula}
Let $J=[v_-,v_+]$ with $v_-<v_+$ and let
$J_0\subset(v_-,v_+)$ be a closed interval. Fix
$\Omega\in C_c^\infty((1/2,2);\C)$ and $A_0,B_0,b_0>0$;
write
$
 \mathfrak p=(J,J_0,\Omega,A_0,B_0,b_0).
$
Suppose $a\in C_c^\infty(\R;\C)$ and
$\Psi\in C^\infty(\R;\R)$ satisfy
\[
 \operatorname{supp}a\subset J_0,\quad
 \|a\|_{C^2(\R)}\le A_0,\quad
 \sup_\R|\Psi^{(s)}|\le B_0\ (1\le s\le3),\quad
 \inf_J|\Psi'|\ge b_0.
\]
For real $M\ge1$, $\tau>0$, and $\xi\in\R$, define
\[
 \mathfrak D_{M,\tau}(a,\Psi;\xi)
 =\sum_{m\in\Z}\Omega(m/M)
        \int_\R a(v)e(\xi\Psi(v)-\tau mv)\,dv.
\]
For each fixed $C_*>0$, uniformly for $\tau\ge1$ and
$|\xi|\le C_*M\tau$,
\begin{equation}
 \begin{split}
 \mathfrak D_{M,\tau}(a,\Psi;\xi)
 =\frac1\tau\sum_{r\in\Z}a(r/\tau)
   \Omega\left(\frac{\xi\Psi'(r/\tau)}{M\tau}\right)
                                  e(\xi\Psi(r/\tau))
 +O_{\mathfrak p,C_*}((M\tau)^{-1}).
 \end{split}
 \label{eqB:poisson-formula}
\end{equation}
Also,
\begin{align}
 \sup_{\xi\in\R}|\mathfrak D_{M,\tau}(a,\Psi;\xi)|
 &\lesssim_{\mathfrak p}1+\tau^{-1}
                            &&(M\ge1,\ \tau>0),
       \label{eqB:absolute}\\
 |\mathfrak D_{M,\tau}(a,\Psi;\xi)|
 &\lesssim_{\mathfrak p}(M\tau)^{-1}
             &&(\tau\ge1,\ |\xi|>C_\infty M\tau),
       \label{eqB:large-frequency}
\end{align}
where $C_\infty=\max\{1,4/b_0\}$. 
\end{lemma}
\begin{proof}
The proof is divided into three steps.
	
\medskip
\noindent\textbf{Step 1. Periodization and the absolute bound.}
Extend $\Omega$ by zero to $\R$ and define the $1$-periodic
trigonometric polynomial
$$B_M(z)=\sum_{m\in\Z}\Omega(m/M)e(-mz)$$ for real $z$.
By Poisson summation,
\begin{equation}
 B_M(z)=M\sum_{r\in\Z}\widehat\Omega(M(z+r))=M\sum_{r\in\Z}\widehat\Omega(M(z-r)).
 \label{eq3-add-periodic-poisson}
\end{equation}
For fixed $M$, rapid decay of $\widehat\Omega$ gives absolute and
locally uniform convergence on the right.
Then
\[
 \begin{split}
 \int_0^1|B_M(z)|\,dz
 \le\sum_{r\in\Z}\int_0^1M|\widehat\Omega(M(z-r))|\,dz
 =\sum_{r\in\Z}\int_{-Mr}^{M(1-r)}|\widehat\Omega(s)|\,ds
 =\|\widehat\Omega\|_{L^{1}(\R)}.
 \end{split}
\]
The interval $\tau J_0$ is covered by at most $\tau|J_0|+2$
unit intervals with integer endpoints. Therefore, by periodicity,
\[
 \begin{split}
 |\mathfrak D_{M,\tau}(a,\Psi;\xi)|
 \le\frac{A_0}{\tau}\int_{\tau J_0}|B_M(z)|\,dz
 \le A_0\left(|J_0|+2/\tau\right)\|\widehat\Omega\|_{L^{1}(\R)}.
 \end{split}
\]
This proves \eqref{eqB:absolute}.


\medskip
\noindent\textbf{Step 2. Taylor expansion with an explicit lattice count.}
Fix $C_{*}>0$ and fix $M\ge 1,\tau\ge 1,|\xi|\le C_{*}M\tau$. 
By \eqref{eq3-add-periodic-poisson}
and the substitution $z=M(\tau v-r)$,
\begin{equation}
\mathfrak D_{M,\tau}(a,\Psi;\xi)
=\frac1\tau\sum_{r\in\Z}\int_\R\widehat\Omega(z)
a\left(\frac{r+z/M}{\tau}\right)
e\left(\xi\Psi\left(\frac{r+z/M}{\tau}\right)\right)\,dz.
\label{eq3-add-dual-exact}
\end{equation}
The sum of the absolute integrals, multiplied by $1/\tau$, is at most
$A_0(|J_0|+2/\tau)\|\widehat\Omega\|_{L^{1}(\R)}$, since a nonzero
summand requires $r\in\tau J_0-z/M$; this justifies the interchange.

For each integer
$r$ and real $z$, write
\[
 s_r=r/\tau,\qquad q_z=z/(M\tau).
\]
Replace the factor
$a(s_r+q_z)e(\xi\Psi(s_r+q_z))$ by
$a(s_r)e(\xi\Psi(s_r)+\xi\Psi'(s_r)q_z)$.
The exact difference is $E_1(r,z)+E_2(r,z)$, where
\begin{align*}
 E_1(r,z)&=(a(s_r+q_z)-a(s_r))e(\xi\Psi(s_r+q_z)),\\
 E_2(r,z)&=a(s_r)\left[e(\xi\Psi(s_r+q_z))
             -e(\xi\Psi(s_r)+\xi\Psi'(s_r)q_z)\right].
\end{align*}
The mean value bound for $a$ and the integral remainder formula
\[
 \Psi(s_r+q_z)-\Psi(s_r)-q_z\Psi'(s_r)
 =q_z^2\int_0^1(1-t)\Psi''(s_r+tq_z)\,dt
\]
give
\[
 |E_1(r,z)|\le A_0|q_z|,\qquad
 |E_2(r,z)|\le 2\pi A_0B_0|\xi|q_z^2.
\]

For fixed $z$, $E_1(r,z)$ is zero unless
$s_r\in J_0\cup(J_0-q_z)$, and $E_2(r,z)$ is zero unless
$s_r\in J_0$. Each of these intervals contains at most
$\tau|J_0|+2$ points of $\tau^{-1}\Z$.
It follows that
\begin{equation}
 \begin{split}
 \sum_{r\in\Z}(|E_1(r,z)|+|E_2(r,z)|)
 \lesssim_{\mathfrak p}(\tau+1)
       \left(\frac{|z|}{M\tau}+\frac{|\xi|z^2}{M^2\tau^2}\right).
 \end{split}
 \label{eq3-add-dual-error}
\end{equation}
Let $$\mu_i(\Omega)=\int_\R|z|^i|\widehat\Omega(z)|\,dz$$ for
$i=1,2$. Both are finite since $\Omega$ is smooth with compact
support. Integrating \eqref{eq3-add-dual-error} against
$|\widehat\Omega(z)|/\tau$ gives an error at most
\[
 C_{\mathfrak p}(1+\tau^{-1})
 \left(\frac{\mu_1(\Omega)}{M\tau}
       +\frac{|\xi|\mu_2(\Omega)}{M^2\tau^2}\right)
 \le\frac{C_{\mathfrak p,C_*}}{M\tau}.
\]

Only finitely many $r$ have $s_r\in J_0$.
Fourier inversion in each of them gives
\[
 \int_\R\widehat\Omega(z)
       e\left(\frac{\xi\Psi'(s_r)}{M\tau}z\right)\,dz
 =\Omega\left(\frac{\xi\Psi'(s_r)}{M\tau}\right).
\]
This proves \eqref{eqB:poisson-formula}.

\medskip
\noindent\textbf{Step 3. The complementary large-frequency range.}
Suppose $|\xi|>C_\infty M\tau$ and $\tau\ge1$.
For each relevant integer $m$, $|m|<2M$, and for $v\in J$,
\[
 |\xi\Psi'(v)-\tau m|
 \ge b_0|\xi|-2M\tau\ge b_0|\xi|/2.
\]
Normalize the phase by setting
\[
 p_{m,\tau,\xi}(v)=\operatorname{sgn}(\xi)\Psi(v)
                          -\frac{\tau m}{|\xi|}v,\qquad v\in\R.
\]
Its first derivative has absolute value at least $b_0/2$ on $J$,
and its second and third derivatives there are bounded by $B_0$.
The amplitude has support in $J_0$ and a $C^2$ bound $A_0$.

Applying \eqref{eq3-add-nonstationary} with $B=2$ and large
parameter $|\xi|>1$ therefore bounds each integral $$\int_\R a(v)e(\xi\Psi(v)-\tau mv)\,dv$$ by
$C_{\mathfrak p}|\xi|^{-2}$. There are at most $3M$ relevant
integers $m$, so
\[
 |\mathfrak D_{M,\tau}(a,\Psi;\xi)|
 \lesssim_{\mathfrak p}M|\xi|^{-2}
 \lesssim_{\mathfrak p}{(M\tau)^{-1}}.
\]
This proves \eqref{eqB:large-frequency} and completes the proof of
\Cref{lem3-add-dual-formula}.
\end{proof}

\section{An arithmetic counting estimate}\label[appendix]{app:resonance}
\Cref{lem3-add-resonance} isolates the counting estimate used in \Cref{prop3:offdiagonal}.
\begin{lemma}\label[lemma]{lem3-add-resonance}
Fix $0<c_1<c_2$ and positive $\tau_*,q_*,B,C_*$, and write
$
 \mathfrak r=(c_1,c_2,\tau_*,q_*,B,C_*).
$
Let $N\ge2$ be an integer and let $Q,A,\tau>0$ satisfy
\[
 Q\ge q_*,\qquad A=\tau Q,\qquad \tau\ge\tau_*,\qquad
 A+Q\le C_*N^B.
\]
Define $\mathcal P_Q=[c_1Q,c_2Q]\cap\N$ and, for $p\in\mathcal P_Q$,
let $d_A(p)=\|A/p\|_{\R/\Z}\in[0,1/2]$.
For every $\epsilon>0$ and $v\in(0,1/2]$,
\begin{align}
 |\{p\in\mathcal P_Q:d_A(p)\le v\}|
 &\lesssim_{\mathfrak r,\epsilon}N^\epsilon(1+Qv),
       \label{eq3-add-resonance-count}\\
 \sum_{\substack{p\in\mathcal P_Q\\d_A(p)>v}}d_A(p)^{-1}
 &\lesssim_{\mathfrak r,\epsilon}
         N^\epsilon\bigl(v^{-1}+Q\log(2/v)\bigr).
       \label{eq3-add-resonance-reciprocal}
\end{align}
\end{lemma}
\begin{proof}
For $p\in\mathcal P_Q$ with $d_A(p)\le v$, choose the nonnegative
nearest integer
\[
 \ell_p(A)=\lfloor A/p+1/2\rfloor.
\]
Then
\[
 |A-p\ell_p(A)|=p\,d_A(p)\le pv\le c_2Qv.
\]
We first count those indices with $\ell_p(A)\ge1$.
$b_p(A):=p\ell_p(A)$ are positive integers in
$[A-c_2Qv,A+c_2Qv]$, an interval containing at most $2c_2Qv+2$
integers. Moreover,
\[
 b_p(A)\le A+c_2Q/2
 \le\max\{1,c_2/2\}(A+Q)
 \le C_1^{\rm size}N^B,
 \qquad C_1^{\rm size}=\max\{1,c_2/2\}C_*.
\]

For each possible product $b$, every eligible $p$ is a positive
divisor of $b$. We may bound their number by the full divisor
function $d_{\rm div}(b)$.
Write $$b=\prod_{\ell\ \textup{prime}}\ell^{v_\ell(b)}$$ and fix $\delta>0$. For prime $\ell$ with $\ell^\delta\ge2$,
\[
 a+1\le2^a\le\ell^{\delta a}\qquad(a\in\Z_{\ge0}).
\]
For each of the finitely many other primes,
$$C_{\ell,\delta}:=\sup_{a\in\Z_{\ge0}}((a+1)\ell^{-\delta a})$$
is finite. Multiplying gives
\[
 d_{\rm div}(b)=\prod_{\ell\ \textup{prime}}(v_\ell(b)+1)
 \le\left(\prod_{\substack{\ell^\delta<2\\ \ell\ \textup{prime}}}C_{\ell,\delta}\right)b^\delta
 =C_\delta b^\delta.
\]
The constant $C_\delta$ depends on $\delta$ alone; the same inequality
includes $b=1$. Choose $\delta=\epsilon/(2B)$. For
$b\le C_1^{\rm size}N^B$,
\[
 d_{\rm div}(b)\le C_\delta(C_1^{\rm size})^\delta N^{\epsilon/2}
                  \lesssim_{\mathfrak r,\epsilon}N^\epsilon.
\]
Multiplication by the number of possible products proves a total
bound $C_{\mathfrak r,\epsilon}N^\epsilon(1+Qv)$ for
$\ell_p(A)\ge1$.

If $\ell_p(A)=0$, then $A/p=d_A(p)\le v$. Hence
\[
 v\ge\frac A p\ge\frac\tau{c_2}\ge\frac{\tau_*}{c_2}.
\]
The total number of such candidates is at most
\[
 |\mathcal P_Q|\le(c_2-c_1)Q+1
 \le(c_2-c_1+q_*^{-1})Q
 \le\frac{c_2(c_2-c_1+q_*^{-1})}{\tau_*}\,Qv.
\]
Adding this bound proves \eqref{eq3-add-resonance-count}.

For \eqref{eq3-add-resonance-reciprocal}, the sum is empty when
$v=1/2$. If $0<v<1/2$, define, for integers $j\ge0$ with
$2^jv<1/2$,
\[
 w_j(v)=\min\{2^{j+1}v,1/2\},\qquad
 E_j(v)=\{p\in\mathcal P_Q:2^jv<d_A(p)\le w_j(v)\}.
\]
These are disjoint sets covering all indices with $d_A(p)>v$,
including $d_A(p)=1/2$. Applying
\eqref{eq3-add-resonance-count} at $w_j(v)\le1/2$ gives
\[
 \begin{split}
 \sum_{p\in E_j(v)}d_A(p)^{-1}
 \le(2^jv)^{-1}C_{\mathfrak r,\epsilon}N^\epsilon
                                    (1+Qw_j(v))
 \le C_{\mathfrak r,\epsilon}N^\epsilon
                                    (2^{-j}v^{-1}+2Q).
 \end{split}
\]
The first terms sum to at most
$2C_{\mathfrak r,\epsilon}N^\epsilon v^{-1}$.
The number of admissible indices $j$ is at most
\[
 1+\log_2(1/(2v))\lesssim \log(2/v).
\]
The second terms therefore sum to at most
$C_{\mathfrak r,\epsilon}N^\epsilon Q\log(2/v)$.
This proves \eqref{eq3-add-resonance-reciprocal} and
\Cref{lem3-add-resonance}.
\end{proof}

\end{document}